\documentclass[12pt,english,reqno]{amsart}
 \usepackage[hypertex,
 colorlinks=true,
 linkcolor=blue,
 citecolor=blue,
 ]{hyperref}
\usepackage[T1]{fontenc}
\usepackage{mathrsfs}
\usepackage{amsmath,amsthm,amssymb}
\usepackage{latexsym}
\usepackage{times}
\usepackage{enumerate}
\usepackage{mathrsfs}
\usepackage{stmaryrd}
\usepackage{amsopn}
\usepackage{amsmath}
\usepackage{amssymb}
\usepackage{amsfonts}
\usepackage{amsbsy}
\usepackage{amscd,indentfirst,epsfig}
\usepackage{hyperref}
\usepackage{amsfonts,amsmath,latexsym,amssymb,verbatim,amsbsy}
\usepackage{amsthm,psfrag}
\usepackage{dsfont}
\usepackage{colordvi}
\usepackage{pstricks}
\renewcommand{\theequation}{\arabic{section}.\arabic{equation}}

\newtheorem{theo}{Theorem}[section]
\newtheorem{lemme}[theo]{Lemma}
\newtheorem{lemmeA}{Lemma A.}

\newtheorem{exaA}{Example A.}

\newtheorem{propo}[theo]{Proposition}
\newtheorem{cor}[theo]{Corollary}
\newtheorem{hyp}[theo]{Assumptions}
\newtheorem{nb}[theo]{Remark}

\newtheorem{exa}[theo]{Example}
\theoremstyle{definition}

\def \leq {\leqslant}
\def \geq {\geqslant}
\numberwithin{equation}{section}
\def\ind#1{\lower5pt\hbox{$\scriptstyle #1$}}
\def \le {\leqslant}
\def \ge {\geqslant}

\def \d {\, \mathrm{d} }

\def \ds {\displaystyle}

\def \ds {\displaystyle}
\def\Q {\mathcal{Q}}

\def\R{{\mathbb R}}
\def \S {{\mathbb S}^2}
\def \E {\mathcal{E}}
\def \n {\widehat{n}}

\def \v {{v}}

\def \vb {\v_{\star}}
\def \w {{w}}

\def \wb {\w_{\star}}

\def \IS {\int_{\S}}
\def \IR {\int_{\R^3}}
\def \IRR {\int_{\R^3 \times \R^3}}

\def\cn{\cdot \n}

\def \F {\mathcal{F}}

\def \us {\widehat{u} \cdot \sigma}

\title[ Haff's law for viscoeslastic hard spheres]
{\textbf{Free cooling and high-energy tails of granular gases with variable restitution coefficient}}

\author{Ricardo J. Alonso \& Bertrand Lods}

\address{\textbf{Ricardo J. Alonso}, Department of Computational and Applied Mathematics, Rice University, Houston, TX 77005-1892.}
\email{ralonso@math.utexas.edu}

\address{\textbf{Bertrand Lods}, Clermont Université, Universit\'{e} Blaise Pascal, Laboratoire de Math\'{e}matiques, CNRS UMR 6620,  BP 10448, F-63000 CLERMONT-FERRAND,
France.}\email{bertrand.lods@math.univ-bpclermont.fr}

\thanks{This work began while both the authors were Core Participants  to the active program of research "Quantum and Kinetic
Transport: Analysis, Computations, and New Application" in residence at \textsc{Institute Of Pure
And Applied Mathematics} (NSF Math. Institute), UCLA, Los Angeles, CA. We thank the organizers of the program for the invitation and the IPAM for excellent working conditions. R. Alonso acknowledges the support from NSF grant DMS-0439872 and ONR grant N000140910290.
}
\begin{document}

\maketitle

\begin{abstract} We prove the so-called generalized Haff's law yielding the optimal algebraic cooling rate of the temperature of a granular gas described by the homogeneous Boltzmann equation for inelastic interactions with non constant restitution coefficient. Our analysis is carried through a careful study of the infinite system of moments of the solution to the Boltzmann equation for granular gases and precise $L^p$ estimates in the self-similar variables.  In the process, we generalize several results on the Boltzmann collision operator obtained recently for homogeneous granular gases with constant restitution coefficient to a broader class of physical restitution coefficients that depend on the collision impact velocity.  This generalization leads to the so-called $L^{1}$-exponential tails theorem for this model.
\end{abstract}

\medskip
\section{Introduction}
\label{intro}
\setcounter{equation}{0}

\subsection{General setting} Rapid granular flows can be successfully described by the Boltzmann equation conveniently modified to account for the energy dissipation due to the inelasticity of collisions. For such a description, one usually considers the collective dynamics of  inelastic
hard-spheres interacting through binary collisions \cite{BrPo,PoSc,Vi}. The loss of mechanical energy due to collisions is characterized by the so-called normal restitution coefficient which quantifies the loss of relative
normal velocity of a pair of colliding particles after the
collision with respect to the impact velocity. Namely, if $v$ and $\vb$  denote the velocities of two particles before they collide, their respective velocities $v'$ and $\vb'$ after collisions are such that
\begin{equation}\label{coef}
(u'\cdot \n)=-(u\cdot \n) \,e,
\end{equation}
where the restitution coefficient $e$ is such that
$0 \leq e \leq 1$ and $\n \in \mathbb{S}^2$  determines the impact direction, i.e. $\n$ stands for the unit vector that points from the $v$-particle center to the $\vb$-particle center at the instant of impact.  Here above
$$u=v-\vb,\qquad u'=v'-\vb',$$
denote respectively the relative velocity before and after collision. The major part of the investigation, at the physical as well as the mathematical levels, has been devoted to the particular case of a constant normal restitution. However, as described in the monograph \cite{BrPo}, it appears that a  more relevant description of granular gases should deal with a variable restitution coefficient $e(\cdot)$ depending on the impact velocity, i.e.
$$e:=e(|u \cdot \n|).$$
The most common model is the one  corresponding to visco-elastic hard-spheres for which the restitution coefficient has been derived by Schwager and P\"{o}schel in \cite{PoSc}. For this peculiar model, $e(\cdot)$  admits the following representation as an infinite expansion series:
\begin{equation}\label{visc}e(|u \cdot \n|)=1+ \sum_{k=1}^\infty (-1)^k a_k |u \cdot \n|^{k/5}, \qquad u \in \R^3, \quad \n \in \mathbb{S}^2\end{equation}
where $a_k \geq 0$ for any $k \in  \mathbb{N}.$ We refer the reader to \cite{BrPo,PoSc} for the physical considerations leading to the above expression (see also the Appendix A  for several properties of $e(\cdot)$ in the case of visco-elastic hard-spheres). This is the principal example we have in mind for most of the results in the paper, though, as we shall see, our approach will cover more general cases including the one of constant restitution coefficient.\medskip

In a kinetic framework, behavior of the granular flows is described, in the spatially situation we shall consider here, by the so-called velocity distribution $f(v,t)$ which represents the probability density of particles with velocity $v \in \R^3$ at time $t \geq 0.$  The time-evolution of the
one-particle distribution function \(f(t,v)\), \(v\in\R^3\),
\(t>0\) satisfies the following
\begin{equation}
\label{be:force} \partial_t f(t,v)  = \Q_e(f,f)(t,v ), \qquad f(t=0,v)=f_0(v)
\end{equation}
where \(\Q_e(f,f)\) is the
inelastic Boltzmann collision operator, expressing the effect of
binary collisions of particles. The collision operator $\Q_e$ shares a common structure with the classical Boltzmann operator for elastic collision \cite{Cer,Vil} but is conveniently modified in order to take into account the inelastic character of the collision mechanism. In particular, $\Q_e$ depends in a very strong and explicit way on the restitution coefficient $e$. Of course, for $e\equiv 1$, one recovers the classical Boltzmann operator. We postpone to Section \ref{Sec21} the precise expression of $\Q_e$. Due to the dissipation of kinetic energy during collisions, in the absence of external forces, the granular temperature
$$\E(t)=\IR f(t,v)|v|^2\d v$$
is continuously decreasing and is expected to go to zero as time goes to infinity, expressing the \textit{cooling of the granular gases}.

Determining the precise rate of decay to zero for the granular temperature is the main goal of the present work. The asymptotic behavior for the granular temperature was first explained in \cite{haff} by P. K. Haff at the beginning of the 80's for the case of constant restitution coefficient, thus, it has become standard to refer to this behavior simply as \textit{Haff's law}.

The mathematical study of Boltzmann models for granular flows was first restricted to the so-called inelastic Maxwell molecules where the collision rate is independent of the relative velocity \cite{BisiCT,BisiCT2,Bobylev-Carrillo-Gamba, ccc, CaTo}.  Later, the mathematical investigation of hard-spheres interactions was initiated in \cite{GaPaVi} for diffusively heated gases and continued in a series of papers \cite{MMR,MiMo} where the first rigorous proof of the Haff's law was presented in the case of constant restitution coefficient.  Additional relevant work in the existence and stability of the homogeneous cooling state can be found in \cite{MiMo2,MiMo3}.  We refer to \cite{Vi} for a mathematical overview of the relevant questions addressed by the kinetic theory of granular gases and complete bibliography on the topic. \smallskip

From the mathematical viewpoint the literature on granular gases with \textit{variable restitution coefficient} is rather limited. However, the Cauchy problem for the homogeneous inelastic Boltzmann equation has been studied in great detail and full generality in \cite{MMR}, including the class of restitution coefficients that we are dealing with in this paper.  For the inhomogeneous inelastic Boltzmann equation the literature is yet more scarce, in this respect we mention the work by one of the authors \cite{AlonsoIumj} that treats the Cauchy problem in the case of near-vacuum data.  It is worthwhile mentioning that the scarcity of results regarding existence of solutions for the inhomogeneous case is explained by the lack of entropy estimates for the inelastic Boltzmann equation, thus, well known theories like the DiPerna-Lions \textit{renormalized solutions} are no longer available.  More complex behavior that involve boundaries, for instance clusters and Maxwell demons, a
 re well beyond of the present techniques.\smallskip

\subsection{Main results and methodology}
Physical considerations and careful dimensional analysis led P. K. Haff \cite{haff} to predict that, for \textit{constant restitution coefficient}, the temperature $\E(t)$ of a granular gas should cool down at a quadratic rate:
$$\E(t)=O\left(t^{-2}\right) \text{ as } t \to \infty.$$
Similar considerations led Schwager and P\"{o}schel \cite{PoSc} to conclude that, for the restitution coefficient associated to the visco-elastic hard-spheres \eqref{visc}, the decay should be slower than the one predicted by Haff, namely at an algebraic rate proportional to $t^{-5/3}$.  These considerations are precisely described in the main result of this paper where the key intuitive fact is that the decay rate of $\E(t)$ is completely determined by the behavior of the restitution coefficient $e(|u \cdot \n|)$ for small impact velocity (Assumption \textit{(1)} in \ref{HYP2}). Precisely, our result is valid for restitution coefficient such that there exist some constants $\alpha >0$ and $\gamma \geq 0$ such that
\begin{equation*}  e(|u \cdot \n|)\simeq 1 - \alpha |u \cdot \n|^\gamma \quad \text{ for } \quad|u \cn| \simeq 0\end{equation*}
and reads as follows:
\begin{theo}\label{haff1} For any initial distribution velocity $f_0 \geq 0$ satisfying the  conditions given by \eqref{initial} with $f_0 \in L^{p_0}(\R^3)$ for some $1 < p_0 < \infty$, the solution $f(t,v)$ to the associated Boltzmann equation \eqref{cauch} satisfies the generalized Haff's law for variable restitution coefficient $e(\cdot)$ fulfilling Assumptions \ref{HYP2} and  \ref{HYPdiff}:
\begin{equation}\label{Haff's}
c (1+t)^{-\frac{2}{1+\gamma}}  \leq \E(t) \leq  C (1+t)^{-\frac{2}{1+\gamma}}, \qquad t \geq 0
\end{equation}
where $\E(t)=\IR f(t,v)|v|^2 \d v$ and $c,C$ are positive constants depending only on $e(\cdot)$ and $\E(0)$.
\end{theo}
We recover with Theorem \ref{haff1} the optimal decay for constant restitution coefficient ($\gamma=0$) given in \cite{MiMo2} and the one predicted for viscoelastic hard-spheres ($\gamma=1/5$) in \cite{PoSc}. The method of the proof has similarities to that of the constant restitution coefficient \cite{MiMo2} but technically more challenging. \medskip

The main tools to prove Theorem \ref{haff1} are the following:
\begin{enumerate}[$\bullet$\:]
 \item The study of the moments of solutions to the Boltzmann equation using a generalization of the Povzner's lemma developed in \cite{BoGaPa}.
 \item Precise $L^p$ estimates, in the same spirit of \cite{MiMo2}, of the solution to the Boltzmann equation for $p>1$.
 \item For the previous item, the analysis is understood in an easiest way using \textit{rescaled solutions} to \eqref{be:force} of the form
$$f(t,v)=V(t)^3 g(\tau(t),V(t)v)$$
where  $\tau(\cdot)$ and $V(\cdot)$ are fixed time-scaling functions to be crafted depending upon the restitution coefficient.  In the \textit{self-similar variables} $(\tau,w)$ the function $g(\tau,w)$ is a solution of an evolution problem of the type
\begin{equation}\label{qsel}\partial_\tau g(\tau,w) +  \xi(\tau) \nabla_w \cdot \left(w g(\tau,w)\right)= {\Q}_{\widetilde{e}(\tau)}(g,g)\end{equation}
for some  $\xi(\tau)$ depending on the time scale $\tau$.   The collision operator ${\Q}_{\widetilde{e}(\tau)}(g,g)$ is associated to a time-dependent restitution coefficient $\widetilde{e}(\tau)$ (see Section \ref{sec:Self} for details).  In this respect we notice that one notable difference with respect to the case of a constant restitution coefficient treated in \cite{MiMo2} is that the rescaled collision operator depends on the (rescaled) time $\tau$, leading to a \textit{non-autonomous problem} for $g$. This is the main reason why the construction of self-similar profile $g$ independent of $\tau$ obtained in \cite{MiMo2} (Homogeneous Cooling State) is not valid for non constant restitution coefficient.
\end{enumerate}
Let us explain in more details our method of proof:
\begin{enumerate}[1.\:]
\item  We start proving in Sections 2 and 3 an upper bound for the decay of the energy.  This shows that, for restitution coefficients satisfying \ref{HYP2}, the cooling of the temperature is\textit{ at least algebraic}. More precisely, under suitable assumptions on the restitution coefficient $e(\cdot)$, we exhibit a convex and increasing mapping $\mathbf{\Psi}_e$ such that
    $$\dfrac{\d}{\d t}\E(t) \leq -\mathbf{\Psi}_e(\E(t)) \qquad \forall t \geq 0,$$
    which leads to an upper bound for $\E(t)$ of the type $$\E(t) \leq C(1+t)^{-\frac{2}{1+\gamma}} \qquad \forall t \geq 0$$
for some positive constant $C >0.$
\item  The lower bound for the free cooling is much more intricate to establish and consists in proving that the cooling rate found above is optimal, i.e., there exists $c>0$ such that
\begin{equation}\label{lowerintro}\E(t) \geq c(1+t)^{-\frac{2}{1+\gamma}} \quad \forall t \geq 0.
\end{equation}
A careful study of the moments of the solution to \eqref{be:force} shows that it suffices to prove a similar algebraic lower bound with some arbitrary rate, i.e. \eqref{lowerintro} will hold if there exists $\lambda >0$ and $c >0$ such that
$$\E(t) \geq c(1+t)^{-\lambda} \qquad \forall t \geq 0.$$
These two points are proved in the last part of Section 3.
\item To prove that the lower bound with some unprescribed rate $\lambda$ holds  we use, as in \cite{MiMo2}, precise $L^p$ estimates ($p>1$) for solutions to \eqref{be:force} in self-similar variables.  We craft a correct time scaling functions $\tau(t)$ and $V(t)$ such that \eqref{lowerintro} is equivalent to $\mathbf{\Theta}(\tau(t)) \geq c$ (here $\mathbf{\Theta}(\cdot)$ denotes the second moment of $g$).  Once this scale is fixed, the function $g(\tau,w)$ satisfies the rescaled Boltzmann equation \eqref{qsel} with $\xi(\tau) \to 0$ as $\tau \to \infty$.  This is a major difference with the constant restitution coefficient case where $\xi(\tau)\equiv1$.  This technical difficulty is overcome proving that the $L^p$-norms of $g(\tau)$ behaves at most polynomially with respect to $\tau$.  For technical reasons which are peculiar to the inelastic interactions, noticed in \cite{AloCar}, we will restrict ourself to study $L^p$-norms in the range $p \in [1,3)$.  The details can b
 e found in Section 5.
\end{enumerate}
The derivation of precise $L^p$ estimates for the solution $g(\tau,w)$ to \eqref{qsel} requires a careful study of the collision operator $\Q_e$ and its regularity properties. We present in Section 4 a full discussion of the regularity and integrability properties of the gain part of the collision operator $\Q^+_{B,e}$ associated to a general collision kernel $B(u,\sigma)=\Phi(|u|)b(\us)$ satisfying Grad's cut-off assumption (see Section 2 for definition).  This Section is divided in five subsections starting with the Carleman representation of the gain operator $\Q^+_{B,e}$.  It is well-known \cite{lions,MoVi,Wennberg,MiMo} that such a representation is essential for the study of regularizing properties of the gain operator $\Q^+_{B,e}$ when smooth assumptions are imposed on the kernel $B(u,\sigma)$.  Our contribution in Sections 4.3 and 4.4 is to extend the existent theory to the inelastic case with variable restitution coefficient.  Since the estimates of Section 4 will be
  applied for solutions written in self-similar variables, we  make sure that such estimates are \textit{independent} of the restitution coefficient. This allows us to overcome the technical problem of the time dependence of the gain operator in the self-similar variables.  Additional convolution-like inequalities \cite{AloCar,MoVi} are derived in subsection 4.2 assuming minimal regularity of the angular kernel $b(\cdot)$.\medskip

The final part of this work is devoted to the proof of propagation of exponential $L^{1}$-tails where the full power of the Povzner's lemma is exploited.  Much of the argument with a minor adaptation is taken from \cite{BoGaPa}.  This important result is presented in the final Section for convenience and not because the machinery of Sections 4 and 5 is needed to prove it.
\begin{theo}[\textbf{$L^{1}$-exponential tails Theorem}]\label{intexpbounds1} Let $B(u,\sigma)=|u|b(\us)$ be the collision kernel with $b(\cdot)$ satisfying \eqref{normalization} and $b(\cdot) \in L^q(\mathbb{S}^2)$ for some $q \geq 1$.  Assume that the variable restitution coefficient $e(\cdot)$ satisfies Assumptions \ref{HYP2}.  Furthermore, assume that $f_0$ satisfies \eqref{initial}, and that there exists $r_0 >0$ such that
\begin{equation*}
\IR f_0(v)\exp\left(r_0|v|\right) \d v <\infty.
\end{equation*}
Then, there exists some $r\leq r_0$ such that
\begin{equation}\label{expboundssol}
\sup_{t\geq 0}\IR f(t,v)\exp\left(rV(t)|v|\right)\d w < \infty.
\end{equation}
The function $V(t)$ is the appropriate scaling, depending solely on the restitution coefficient, given in \eqref{V(t)}.
\end{theo}

\subsection{Notations} Let us introduce the notations we shall use
in the sequel. Throughout the paper we shall use the notation
$\langle \cdot \rangle = \sqrt{1+|\cdot|^2}$. We denote, for any
$\eta \in \R$, the Banach space
\[
     L^1_\eta = \left\{f: \R^3 \to \R \hbox{ measurable} \, ; \; \;
     \| f \|_{L^1_\eta} := \int_{\R^3} | f (v) | \, \langle v \rangle^\eta \d\v
     < + \infty \right\}.
\]

More generally we define the weighted Lebesgue space $L^p_\eta
(\R^3)$ ($p \in [1,+\infty)$, $\eta \in \R$) by the norm
$$\| f \|_{L^p_\eta(\R^3)} = \left[ \int_{\R^3} |f (v)|^p \, \langle v
        \rangle^{p\eta} \d\v \right]^{1/p} \qquad 1 \leq p < \infty$$
        while $\| f \|_{L^\infty_\eta(\R^3)} =\mathrm{ess-sup}_{v \in \R^3} |f(v)|\langle v\rangle^\eta$ for $p=\infty.$

For any $k \in \mathbb{N}$, we denote by $H^k=H^k(\R^3)$ the usual Sobolev space defined by
the norm
 \begin{equation*}\label{sobnorm} \| f \|_{H^{k}}  =   \left[ \sum_{|j| \le k} \|\partial_v^j f\|_{L^2} ^p \right]^{1/p} \end{equation*}
where $\partial_v^j$ denotes the partial derivative associated with
the multi-index $j \in \mathbb{N}^N$. Moreover this definition can be
extended to $H^s$ for any $s \ge 0$ by using the Fourier
transform $\mathcal{F}.$  The binomial coefficients for non-integer $p \geq 0$ and $k \in \mathbb{N}$ are
defined as
$$\left(\begin{array}{c}
p\\k
\end{array}\right)=\dfrac{p(p-1)\ldots(p-k+1)}{k!} \quad k \geq 1, \quad \left(\begin{array}{c}
p\\0
\end{array}\right)=1.$$

\section{Preliminaries}

\subsection{The kinetic model}\label{Sec21}

We assume the granular particles to be perfectly smooth hard-spheres of mass
$m=1$ performing inelastic collisions. Recall that, as explained in the Introduction, the inelasticity of the collision mechanism is characterized by a single
parameter, namely the coefficient of normal restitution \(0 \leq
e \leq 1\) which we assume  to be \textit{non constant}. More precisely, let $(v,\vb)$ denote the velocities of two particles before they collide. Their respective velocities after collisions $v'$ and $\vb'$ are given, in virtue of \eqref{coef} and the conservation of momentum, by
\begin{equation}
\label{transfpre}
  v'=v-\frac{1+e}{2}\,(u\cdot \n)\n,
\qquad \vb'=\vb+\frac{1+e}{2}\,(u\cdot \n)\n,
\end{equation}
where the symbol $u$ stands for the relative velocity $u=v-\vb$ and $\n$ is the impact direction.  From the physical viewpoint, a common approximation consists in choosing $e$ as a suitable function of the impact velocity, i.e. $e:=e(|u \cdot \n|)$. The main assumptions on the function $e(\cdot)$ are listed in the following (see \cite{AlonsoIumj}):
\begin{hyp}\label{HYP}
Assume the following hold:
\begin{enumerate}
\item The mapping  $r \in \mathbb{R}_+ \mapsto e(r) \in (0,1]$ is absolutely continuous.
\item The mapping $r\in\mathbb{R}_{+}\rightarrow \vartheta(r):=r\;e(r)$ is strictly increasing.
\end{enumerate}
\end{hyp}
Further assumptions on the function $e(\cdot)$ shall be needed later on. Given assumption $\textit{(2)}$, the Jacobian of the transformation \eqref{transfpre} can be computed as
$$J:=\left|\dfrac{\partial (v',\vb')}{\partial (v,\vb)}\right|=|u\cdot \n| + |u\cdot \n|\dfrac{\d e}{\d r}(|u\cdot \n|)=\dfrac{\d \vartheta}{\d r}(|u\cdot \n|) >0.$$ In practical situations, the restitution coefficient $e(\cdot)$ is usually chosen among the following three examples:
\begin{exa}[\textbf{Constant restitution coefficient}]\label{cons} The most documented example in the literature  is the one in which $$e(r)=e_0 \in (0,1] \qquad \text{ for any } r \geq 0.$$\end{exa}

\begin{exa}[\textbf{Monotone decreasing}]\label{dec} A second example of interest is the one in which the restitution coefficient $e(\cdot)$ is a monotone decreasing function:
\begin{equation}\label{eta}
e(r)=\dfrac{1}{1+a r^\eta}\qquad \forall r \geq 0\end{equation} where $a>0,$ $\eta >0$ are two given constants.\end{exa}

\begin{exa}[\textbf{Viscoelastic hard-spheres}]\label{exa:visco} This is the most physically relevant model treated in this work.  For such a model, the properties of the restitution coefficient have been derived in \cite{BrPo,PoSc} where representation \eqref{visc} is given.  It also accepts the implicit representation
   \begin{equation}\label{visco}e(r) + a r^{1/5} e(r)^{3/5}=1\end{equation}
where $a>0$ is a suitable positive constant depending on the material viscosity (see Figure 1).\end{exa}
\begin{figure} 
 \begin{center}
\psfrag{u}{$r$}\psfrag{v}{$e$}
\includegraphics{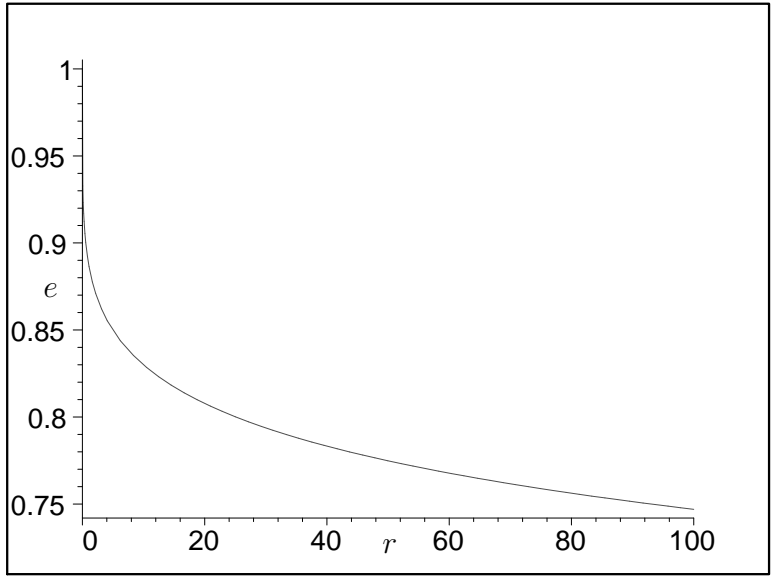}
  \caption{Restitution coefficient for viscoelastic hard-spheres given by Eq. \eqref{visco} with $a=0.12$.}
 \end{center}
\end{figure}
In the sequel, it shall be more convenient to use the following equivalent parametrization of the post-collisional velocities. For distinct velocities $v$ and $\vb$, let $\widehat{u}=\frac{u}{|u|}$ be the relative velocity unit vector. The change of variables
$$\sigma=\widehat{u}-2 \,(\widehat{u}\cdot \n)\n \in\mathbb{S}^2$$
provides an alternative parametrization of the unit sphere $\mathbb{S}^2$ for which the impact velocity reads
$$|u\cdot\n|=|u| \,|\widehat{u} \cdot \n|=|u| \sqrt{\frac{1-\widehat{u} \cdot \sigma}{2}}.$$
Then,  the post-collisional velocities $(v',\vb')$ given in \eqref{transfpre} are transformed to
\begin{equation}
\label{postsig} v'=v-\beta\frac{u-|u|\sigma}{2}, \qquad \vb'=\vb+ \beta\frac{u-|u|\sigma}{2}
\end{equation}
where
$$\beta=\beta\left(|u| \sqrt{\tfrac{1-\widehat{u} \cdot \sigma}{2}}\right)=\frac{1+e}{2} \in \left(\tfrac{1}{2},1\right].$$
In this representation, the \textit{weak formulation} of the Boltzmann collision operator $\Q_{B,e}$ given a collision kernel $B(u,\sigma)$ reads
\begin{equation}\label{Ie3B}
\int_{\mathbb{R}^{3}}\Q_{B,e}(f,g)(v)\psi(v)\d v=\frac{1}{2}\int_{\mathbb{R}^{3} \times \mathbb{R}^3}f(v)g(\vb)\mathcal{A}_{B,e}[\psi](v,\vb)\;\d\vb\d v
\end{equation}
for any suitable test function $\psi=\psi(v)$.  Here
$$\mathcal{A}_{B,e}[\psi](v,\vb)=\int_{\mathbb{S}^{2}}\bigg(\psi( {v'})+\psi( {\vb'})-\psi(v)-\psi(\vb)\bigg) {B}(u, \sigma)\d\sigma$$
with $v',\vb'$ defined in \eqref{postsig}.  We assume that the collision kernel $B(u,\sigma)$ takes the form
$$ {B}(u,\sigma)=\Phi(|u|)b(\widehat{u} \cdot \sigma)$$
where $\Phi(\cdot)$ is a suitable nonnegative function known as \textit{potential}, while the \textit{angular kernel} $b(\cdot)$ is usually assumed belonging to $L^{1}(-1,1)$.  For any fixed vector $\widehat{u}$, the angular kernel defines a measure on the sphere through the mapping $ \sigma \in \mathbb{S}^{2}\mapsto b(\widehat{u}\cdot\sigma)\in[0,\infty]$ that we assume to satisfy the renormalized \textit{Grad's cut-off} hypothesis
\begin{equation}\label{normalization}
\left\|b\right\|_{L^{1}(\mathbb{S}^{2})}=2\pi\left\|b\right\|_{L^{1}(-1,1)}=1.
\end{equation}
The most relevant model in our case is hard-spheres which correspond to $\Phi(|u|)=|u|$ and $b(\widehat{u} \cdot \sigma)=\frac{1}{4\pi}$. We shall consider the \textbf{\textit{generalized hard-spheres collision kernel}} for which $\Phi(|u|)=|u|$ and the angular kernel is non necessarily constant but satisfying \eqref{normalization}. For the particular model of hard-spheres interactions, we simply denote the collision operator $\Q_{B,e}$ by $\Q_e$.

\subsection{On the Cauchy problem} We consider the following homogeneous Boltzmann equation
\begin{equation}\label{cauch}\begin{cases}
\partial_t f(t,v)&=\Q_{B,e}(f,f)(t,v) \qquad \qquad t >0, \; v \in \mathbb{R}^3\\
f(0,v)&=f_0(v), \qquad \qquad  v \in \mathbb{R}^3
\end{cases}
\end{equation}
where the initial datum $f_0$ is a \textit{nonnegative} velocity function such that
\begin{equation}\label{initial}\IR f_0(v)\d v=1, \quad \IR f_0(v) v \d v =0 \quad \text{ and } \quad \IR f_0(v)|v|^3 \d v < \infty.\end{equation}
There is no loss of generality in assuming the two first moments conditions in \eqref{initial} due to scaling and translational arguments.
We say that a nonnegative $f=f(t,v)$ is a solution to \eqref{initial} if $f \in \mathcal{C}([0,\infty),L^1_2(\R^3))$ and
$$\int_0^\infty \d t \IR \bigg(f(t,v) \partial_t \psi(t,v) + \Q_{B,e}(f,f)(t,v)\psi(t,v) \bigg)\d v = \IR f_0(v) \psi(0,v)\d v$$
holds for any compactly supported $\psi \in \mathcal{C}^1([0,\infty) \times \R^3).$
 Under the Assumptions \ref{HYP}, the assumptions \textbf{H1} and \textbf{H2} of \cite{MiMo} are fulfilled (with the terminology of \cite{MiMo}, we are dealing with a  non-coupled collision rate and, more precisely, with the so-called \textit{generalized visco-elastic model}, see \cite{MiMo}, p. 661). In particular, \cite[Theorem 1.2]{MiMo} applies direclty and allows us to state:
\begin{theo}[\textbf{Mischler et \textit{al.}}]
For any  nonnegative velocity function $f_0$ satisfying \eqref{initial}, there is a unique solution $f=f(t,v)$ to \eqref{cauch}. Moreover,
\begin{equation}\label{firstmom}\IR f(t,v)\d v =1, \qquad \IR f(t,v) v \d v =0 \qquad \forall t \geq 0.\end{equation}
\end{theo}

\subsection{Self-similar variables}\label{sec:Self} Let us discuss precisely the rescaling using self-similar variables.  Let $f(t,v)$ be the solution to
\eqref{cauch} associated to some initial datum $f_0$ satisfying \eqref{initial} and collision kernel
$$B(u,\sigma)=\Phi(|u|) b(\widehat{u} \cdot \sigma)$$
with $b(\cdot)$ satisfying \eqref{normalization}.  The rescaled solution $g=g(\tau,w)$ is defined such that
\begin{equation}\label{resca}
f(t,v)=V(t)^3 g(\tau(t),V(t)v)
\end{equation}
where $\tau(\cdot)$ and $V(\cdot)$ are time-scaling functions to be determined solely on the behavior of the restitution coefficient in the low impact velocity region.  Since these are scaling functions they are increasing and satisfy   $\tau(0)=0$ and $V(0)=1$. One has
$$1=\IR f(t,v)\d v=\IR g(\tau(t),w)\d w \qquad \forall t \geq 0$$
and $g(0,w)=f_0(w)$.
Furthermore, some elementary calculations show that the function $g(\tau,w)$ satisfies
\begin{equation}\label{scal}V(t)^{-2} \Q_e(f,f)(t,v) =
\dot{\tau}(t) {V} (t)\partial_\tau g(\tau,w) +
\dot{V}(t) \nabla_w \cdot (w g(\tau,w))\bigg|_{\underset{\tau=\tau(t)}{w=V(t)v}}\end{equation}
where the dot symbol denotes the derivative with respect to $t$. Moreover, the expression of the collision operator in the self-similar variables is
$$  V(t)^{-2}\Q_{B,e}(f,f)\left(t,\dfrac{v}{V(t)}\right)=\Q_{B_{\tau},\widetilde{e}_{\tau}}(g,g)(\tau(t),v)
$$
where the rescaled collision kernel $B_\tau$ is given by $$B_{\tau(t)}(u,\sigma):=V(t)\Phi\left(\dfrac{|u|}{V(t)}\right)b(\widehat{u} \cdot
\sigma).
$$
The rescaled restitution coefficient $\widetilde{e}_\tau$ has been defined by
$$
\widetilde{e}_\tau\::\:(r,t)\longmapsto\widetilde{e}_{\tau(t)}(r):=e\left(\frac{r}{V(t)}\right)\ \ \ \mbox{for}\ \ r\geq0,\ \ t\geq0.
$$
Since the mapping $t \in \R^+ \longmapsto \tau(t) \in \R^+$ is injective with inverse $\zeta$, one can rewrite equation \eqref{scal} in terms of $\tau$ only.  Thus, $g(\tau,w)$ is a solution to the following rescaled Boltzmann equation:
\begin{equation}\label{eqgt}
 \lambda(\tau)\partial_\tau g(\tau,w) +
\xi(\tau) \nabla_w \cdot (w g(\tau,w)) =\Q_{B_{\tau},\widetilde{e}_{\tau}}(g,g)(\tau,w) \qquad \forall \tau >0
\end{equation}
with $$\lambda(\cdot)=\dot{\tau}(\zeta(\cdot))V(\zeta(\cdot)) \quad  \text{ and } \quad \xi(\cdot)=\dot{V}(\zeta(\cdot)),$$ and model parameters
\begin{equation}\label{eqB}
B_\tau(u,\sigma)=V(\zeta(\tau))\Phi\left(\dfrac{|u|}{V(\zeta(\tau))}\right)b(\widehat{u}\cdot\sigma)\ \ \ \mbox{and}\ \ \ \widetilde{e}_{\tau}(r)=e\left(\dfrac{r}{V(\zeta(\tau))}\right).
\end{equation}
Notice that, for generalized hard-spheres interactions (i.e. whenever
$\Phi(|u|)=|u|)$ one has $B_\tau=B$. For true hard-spheres interactions, i.e. $b(\cdot)=\frac{1}{4\pi}$, one simply denotes the rescaled collision operator by $\Q_{\widetilde{e}_\tau}$.  In addition, observe that the rescaled operator depends on time, and therefore, $g$ is a solution to a non-autonomous problem.

\subsection{Povzner-type inequalities}\label{sec:povzner} We extend in this section the results of \cite{BoGaPa} and \cite{MiMo2} to the case of variable restitution coefficient satisfying \ref{HYP}.  We consider a collision kernel of the form
$$
B(u,\sigma)=\Phi(|u|)b(\widehat{u}\cdot \sigma),
$$
with angular kernel $b(\cdot)$ satisfying the renormalized Grad's cut-off assumption \eqref{normalization}.  Let $f$ be a \textit{nonnegative} function satisfying \eqref{firstmom} and $\psi(v)=\Psi(|v|^2)$ be a given test-function with $\Psi$ nondecreasing and convex. Then, Eq. \eqref{Ie3B} leads to
$$
\int_{\mathbb{R}^{3}}\Q_{B,e}(f,f)(v)\psi(v)\d v=\frac{1}{2}\int_{\mathbb{R}^{3} \times \mathbb{R}^3}f(v)f(\vb) \mathcal{A}_{B,e}[\psi](v,\vb)\;\d\vb\d v
$$
with
$$
\mathcal{A}_{B,e}[\psi](v,\vb)=\Phi(|u|)\left(A^+_{B,e}[\Psi](v,\vb)-A^-_{B,e}[\Psi](v,\vb)\right)
$$
where
$$
A_{B,e}^+[\Psi](v,\vb)=\int_{\mathbb{S}^{2}}\left(\Psi( |{v'}|^2)+\Psi( |{\vb'}|^2)\right)b(\widehat{u}\cdot\sigma)\d\sigma.
$$
Using \eqref{normalization} we also have,
$$
A_{B,e}^-[\Psi](v,\vb)=\int_{\mathbb{S}^{2}} \left(\psi( {v })+\psi( {\vb })\right)b(\widehat{u}\cdot\sigma)\d\sigma= \left(\Psi( |v |^2)+\Psi( |\vb|^2)\right).
$$
Following \cite{BoGaPa}, we define  the velocity of the center of mass $U=\dfrac{v+\vb}{2}$ so that
$$
v'=U+\dfrac{|u|}{2}\omega, \qquad \vb'=U-\dfrac{|u|}{2}\omega \qquad \text{ with } \quad \omega=(1-\beta)\widehat{u}+\beta \sigma.
$$
Recall that for any vector $x \in \R^3$, we set $\widehat{x}=\frac{x}{|x|}$.  When $e$, or equivalently $\beta$, is constant, the strategy of \cite{BoGaPa} consists, roughly speaking, in performing a suitable change of unknown $\sigma \to \widehat{\omega}$ to carefully estimate $A^+_{B,e}[\psi]$.  For variable $\beta$, such strategy does not apply directly.  Instead, observe that $|\omega| \leq 1$ and, since $\Psi$ is increasing, one has
\begin{equation*}\begin{split}
\Psi( |{v'}|^2)+\Psi( |{\vb'}|^2) &\leq \Psi\left(|U|^2+\frac{|u|^2}{4} + |u| |U| \widehat{U} \cdot \omega\right) + \Psi\left(|U|^2+\frac{|u|^2}{4}- |u| |U| \widehat{U} \cdot \omega\right)\\
&=\Psi\left(E \frac{1+\xi\;\widehat{U}\cdot\omega}{2} \right)+\Psi\left(E \frac{1-\xi\;\widehat{U}\cdot\omega}{2} \right)
\end{split}\end{equation*}
where we have set  $E:=|v|^{2}+|\vb|^{2}=2|U|^{2}+\frac{|u|^2}{2}$ and $\xi=2\frac{|U|\,|u|}{E}$. Since $\Psi(\cdot)$ is convex the mapping $$\Psi_0(t)=\Psi(x+ty)+\Psi(x-ty)$$ is even and nondecreasing for $t\geq0$ and $x,y\in\mathbb{R}$ (see \cite{BoGaPa}).  Therefore, using that $\xi \leq 1$ one gets
\begin{equation}\label{import}
\Psi( |{v'}|^2)+\Psi( |{\vb'}|^2) \leq \Psi\left(E \frac{1+\widehat{U}\cdot\omega}{2} \right)+\Psi\left(E \frac{1-\widehat{U}\cdot\omega}{2} \right).
\end{equation}
In the case that $\widehat{U}\cdot\sigma\geq0$ it follows that
$$
\left|\widehat{U}\cdot\omega\right|=\left|(1-\beta)\widehat{U}\cdot\widehat{u}+\beta\widehat{U}\cdot\sigma\right|\leq(1-\beta)+\beta\widehat{U}\cdot\sigma,
$$
thus, using the fact that $\Psi_0(t)$ is even and nondecreasing for $t\geq0$, we conclude from \eqref{import} that
\begin{equation*}
\Psi( |{v'}|^2)+\Psi( |{\vb'}|^2) \leq \Psi\left(E \frac{2-\beta+\beta\widehat{U}\cdot\sigma}{2}\right)+\Psi\left(E \frac{\beta-\beta\widehat{U}\cdot\sigma}{2} \right).
\end{equation*}
When $\widehat{U}\cdot\sigma\leq0$ a similar argument shows that
\begin{equation*}
\Psi( |{v'}|^2)+\Psi( |{\vb'}|^2) \leq \Psi\left(E \frac{2-\beta-\beta\widehat{U}\cdot\sigma}{2}\right)+\Psi\left(E \frac{\beta+\beta\widehat{U}\cdot\sigma}{2} \right).
\end{equation*}
Hence, setting $\tilde{b}(s)=b(s)+b(-s)$ and using these last two estimates with the change of variables $\sigma\rightarrow-\sigma$ we get
\begin{align}
A_{B,e}^+[\Psi](v,\vb)&\leq \int_{\{\widehat{U}\cdot\sigma\geq0\}} \left[\Psi\left(E \frac{2-\beta+\beta\widehat{U}\cdot\sigma}{2}\right)+\Psi\left(E \frac{\beta-\beta\widehat{U}\cdot\sigma}{2} \right)\right]\tilde{b}(\widehat{u}\cdot\sigma)\d\sigma\nonumber\\
&\leq \int_{\{\widehat{U}\cdot\sigma\geq0\}} \left[\Psi\left(E \frac{3+\widehat{U}\cdot\sigma}{4}\right)+\Psi\left(E \frac{1-\widehat{U}\cdot\sigma}{4} \right)\right]\tilde{b}(\widehat{u}\cdot\sigma)\d\sigma,\label{Ae+}
\end{align}
where the second inequality can be shown writing
\begin{multline*}
\frac{2-\beta+\beta\widehat{U}\cdot\sigma}{2}=\frac{1}{2}+\left(\frac{1}{2}-\frac{\beta}{2}\left(1-\widehat{U}\cdot\sigma\right)\right)\ \ \ \mbox{and},\\
\frac{\beta-\beta\widehat{U}\cdot\sigma}{2}=\frac{1}{2}-\left(\frac{1}{2}-\frac{\beta}{2}\left(1-\widehat{U}\cdot\sigma\right)\right).
\end{multline*}
The term in parenthesis is maximized when $\beta=1/2$, thus the monotonicity of $\Psi_0$ implies the result.\medskip

Next, we particularize the previous estimates to the important case $\Psi(x)=x^p$.  This choice will lead to the study of the moments of solutions.
\begin{lemme}\label{gamm}
 Let $q\geq 1$ be such that  $b\in L^{q}(\mathbb{S}^{2})$.  Then, for any restitution coefficient $e(\cdot)$ satisfying Assumptions \ref{HYP} and any real $p \geq 1$, there exists an explicit constant $\kappa_p>0$ such that
\begin{multline}\label{ABE}
\Phi(|u|)^{-1}\mathcal{A}_{B,e}[|\cdot|^p](v,\vb) \leq -(1-\kappa_p)\left(|v|^{2p}+|\vb|^{2p}\right)
\\+\kappa_p \left[\left(|v|^2+|\vb|^2\right)^p-|v|^{2p}-|\vb|^{2p}\right].
\end{multline}
This constant $\kappa_p$ has the following properties:
\begin{itemize}
\item [(1)] $\kappa_1\leq 1$.
\item [(2)] For $p\geq1$ the map $p \mapsto \kappa_p$ is strictly decreasing. In particular, $\kappa_p<1$ for $p>1$.
\item [(3)] $\kappa_p=O\left(1/p^{1/q'}\right)$ for large $p$, where $1/q+1/q'=1$.
\item [(4)] For $q=1$, one still has $\kappa_p\searrow0$ as $p\rightarrow\infty$.
\end{itemize}
\end{lemme}
\begin{proof}
Let $\Psi_p(x)=x^p$. From \eqref{Ae+}, one sees that $$A_{B,e}^+[\Psi_p](v,\vb)\leq\kappa_p\;E^{p}$$
  where we recall that $E=|v|^{2}+|\vb|^{2}$ and we set
\begin{equation}\label{gammap}\kappa_p=\sup_{\widehat{U},\widehat{u}}\int_{\widehat{U}\cdot \sigma \geq 0} \left[\Psi_p\left(\dfrac{3+\widehat{U}\cdot \sigma}{4}\right)  + \Psi_p\left(\dfrac{1-\widehat{U}\cdot\sigma}{4}\right) \right]\tilde{b}(\widehat{u}\cdot\sigma)\d\sigma.\end{equation}
It is clear that the above inequality yields \eqref{ABE}. Let us prove that $\kappa_p$ satisfies the aforementioned conditions.  First, we use H\"{o}lder inequality to obtain
$$\kappa_p\leq  4\pi\left\|b\right\|_{L^{q}(\mathbb{S}^{2})} \left(\int_{-1}^1\left[\Psi_p\left( \frac{3+s}{4}\right)+\Psi_p\left( \frac{1-s}{4} \right)\right]^{q'}\d s\right)^{1/q'}<\frac{16\pi\left\|b\right\|_{L^{q}(\mathbb{S}^{2})}}{(q'p+1)^{1/q'}}.$$
This proves that $\kappa_p$ is finite and also yields item (3) for $q>1$.  For items (1) and (2) observe that the integral in the right-hand-side \eqref{Ae+} is continuous in the vectors $\widehat{U},\;\widehat{u}\in \mathbb{S}^{2}$.  This can be shown by changing the integral to polar coordinates.  Thus, the supremum in these arguments is achieved.  Therefore, there exist $\widehat{U}_0,\;\widehat{u}_0\in \mathbb{S}^{2}$ (depending on the angular kernel $b$) such that
\begin{equation*}
\kappa_p= \int_{\{\widehat{U}_0\cdot\sigma\geq0\}} \left[\Psi_p\left( \frac{3+\widehat{U}_0\cdot\sigma}{4}\right)+\Psi_p\left(\frac{1-\widehat{U}_0\cdot\sigma}{4} \right)\right]\tilde{b}(\widehat{u}_0\cdot\sigma)\d\sigma.
\end{equation*}
A simple computation with this estimate shows that $\kappa_1=\left\|b\right\|_{L^{1}(\mathbb{S}^{2})}=1$.  Moreover, the integrand is a.e. strictly decreasing as $p$ increases and this proves (2).  Finally, let $p\rightarrow\infty$ in this expression and use Dominated convergence to conclude (4) for the case $q=1$.
\end{proof}
The above lemma is the analogous of \cite[Corollary 1]{BoGaPa} for variable restitution coefficient $e(\cdot)$ and it proves that the subsequent results of \cite{BoGaPa} extend readily to variable restitution coefficient. In particular,  \cite[Lemma 3]{BoGaPa} reads\footnote{Notice that, though stated for hard-spheres interactions only, \cite[Lemma 3]{BoGaPa} applies to our situation thanks to the above Lemma \ref{gamm} and \cite[Lemma 1]{BoGaPa}.}:
\begin{propo}\label{povzner} Let $f$ be a nonnegative function satisfying \eqref{firstmom}.  For any $p \geq 1$, we set
$$
m_p=\IR f(v)|v|^{2p}\d v.
$$
Assume that the collision kernel $B(u,\sigma)=|u|b(\widehat{u}\cdot \sigma)$ is such that $b(\cdot)$ satisfies \eqref{normalization} with $b(\cdot) \in L^q(\mathbb{S}^2)$ for some $q \geq 1.$ For any restitution coefficient $e(\cdot)$ satisfying Assumptions \ref{HYP} and any real $p \geq 1$, one has
\begin{equation}\label{QepSp}
\IR \Q_{B,e}(f,f)(v)|v|^{2p}\d v\leq-(1-\kappa_{p})m_{p+1/2}+\kappa_{p}\;S_{p},
\end{equation}
where,
\begin{equation*}
S_{p}=\sum^{[\frac{p+1}{2}]}_{k=1}\left(
\begin{array}{c}
p\\k
\end{array}
\right)\left(m_{k+1/2}\;m_{p-k}+m_{k}\;m_{p-k+1/2}\right),
\end{equation*}
$[\frac{p+1}{2}]$ denoting the integer part of $\frac{p+1}{2}$ and $\kappa_p$ being the constant of Lemma \ref{gamm}. 
 \end{propo}

Inequality \eqref{QepSp} was introduced in \cite{BoGaPa} because the term $S_p$ involves only moments of order $p-1/2$.   Thus, the above estimate has important consequences on the propagation of moments for the solution to \eqref{cauch} (see Section \ref{sec:haff} for more discussion).

\section{Free cooling of granular gases: generalized Haff's law}\label{sec:haff}

We investigate in this section the so-called generalized \textit{Haff's law for  granular gases with variable restitution coefficient}. More precisely, we aim to derive the exact rate of decay of the temperature $\E(t)$ of the solution to Eq. \eqref{cauch}. In this section, we \textit{exclusively} study the generalized hard-spheres collision kernel.
$$
B(u,\sigma)=|u|b(\widehat{u}\cdot \sigma)
$$
where $b(\cdot)$ satisfies \eqref{normalization} but generalization to the so-called variable hard-spheres interactions (i.e. $\Phi(|u|)=|u|^s$ for $s \geq 0$) is easy to handle. Let $f_0$ be a nonnegative velocity distribution satisfying \eqref{initial} and let $f(t,v)$ be the associated solution to the Cauchy problem \eqref{cauch}. We denote its temperature by $\E(t)$,
$$
\E(t)=\IR f(t,v)|v|^2 \d v.
$$
The conditions \eqref{initial} implies that $\sup_{t\geq 0} \E(t)<\infty.$ Indeed, the evolution of $\E(t)$ is  governed by
\begin{multline*}\dfrac{\d }{\d t}\E(t)=\IR \Q_{B,e}(f,f)(t,v)|v|^2\d v=\frac{1}{2}\int_{\mathbb{R}^{3} \times \mathbb{R}^3}f(t,v)f(t,\vb)|u|\times \\
\times \int_{\mathbb{S}^{2}}\bigg(|v'|^2+|\vb'|^2-|v|^2-|\vb|^2\bigg)b(\widehat{u}\cdot \sigma)\d\sigma\;\d\vb\d v
\end{multline*}
where we applied \eqref{Ie3B} with $\psi(v)=|v|^2$.  One checks readily that
$$
|v'|^2+|\vb'|^2-|v|^2-|\vb|^2=-|u|^2\dfrac{1-\us}{4}\left(1- e^2\left(|u|\sqrt{\frac{1-\us}{2}}\right)\right),
$$
so that
\begin{multline*}
\dfrac{\d }{\d t}\E(t)=-\frac{1}{2}\IRR f(t,v)f(t,\vb)|u|^3\d v\d\vb  \\
\times \IS \dfrac{1-\us}{4}\left(1- e^2\left(|u|\sqrt{\frac{1-\us}{2}}\right)\right)b(\us)\d\sigma.
\end{multline*}
We compute this last integral over $\mathbb{S}^2$ (for fixed $v$ and $\vb$) using polar coordinates to get
\begin{multline*}|u|^3\IS \dfrac{1-\us}{8}\left(1- e^2\left(|u|\sqrt{\frac{1-\us}{2}}\right)\right)b(\us)\d\sigma=\\
2\pi|u|^3 \int_0^1 \left(1-e^2(|u|y)\right) b(1-2y^2) y^3\d y=\mathbf{\Psi}_e(|u|^2)\end{multline*}
where we have defined
\begin{equation}\label{Psie}
\mathbf{\Psi}_e(r):=2\pi r^{3/2}\int_0^{1} \left(1-e(\sqrt{r}z)^2\right)b\left(1-2z^2\right)z^3\d z, \qquad \qquad \forall r\geq0.
\end{equation}
In other words, the evolution of the temperature $\E(t)$ is given by
$$\dfrac{\d }{\d t}\E(t)=-\IRR f(t,v)f(t,\vb)\mathbf{\Psi}_e(|u|^2)\d v\d\vb\leq0, \qquad t \geq 0.$$
In addition to Assumptions \ref{HYP}, we assume in the rest of the paper that the restitution coefficient $e(\cdot)$ satisfies the following:
\begin{hyp}\label{HYP2} Assume that the mapping $r \mapsto e(r) \in (0,1]$ satisfies Assumptions \ref{HYP} and
\begin{enumerate}
\item there exist $\alpha >0$ and $\gamma \geq 0$ such that
$$e(r) \simeq 1-\alpha\, r^\gamma \quad \text{ for } \quad r \simeq 0,$$
\item $\liminf_{r \to \infty} e(r) =e_0 <1,$
\item $b(\cdot) \in L^q(\mathbb{S}^2)$ for some $q \geq 1$, and
\item the function $r>0 \longmapsto \mathbf{\Psi}_e(r)$ defined in \eqref{Psie} is strictly increasing and convex over $(0,+\infty)$.
\end{enumerate}
\end{hyp}
\begin{nb} For hard-spheres interactions, $b(\us)=\frac{1}{4\pi}$, thus, $\mathbf{\Psi}_e$ reduces to
$$\mathbf{\Psi}_e(r)=\dfrac{1}{2\sqrt{r}}\int_0^{\sqrt{r}} \left(1-e(y)^2\right)y^3\d y, \qquad \qquad r >0.$$
We prove in the Appendix that Assumptions \ref{HYP2} are satisfied for the viscoelastic hard-spheres of Example \ref{exa:visco} with  $\gamma=1/5$.  More generally, in the case of hard-spheres interactions, assumption \textit{(4)} is fulfilled if $e(\cdot)$ is continuously decreasing (see Lemma A. \ref{decreas} in Appendix A).  For constant restitution coefficient $e(r)=e_0$, these assumptions are trivially satisfied.
\end{nb}

\subsection{Upper bound for $\E(t)$} We first prove the first half of Haff's law, namely, the temperature $\E(t)$ has at least algebraic decay.
\begin{propo}\label{prop:cool} Let $f_0$ be a nonnegative velocity distribution satisfying \eqref{initial} and let $f(t,v)$ be the associated solution to the Cauchy problem \eqref{cauch} where the variable restitution coefficient satisfies Assumptions \ref{HYP2}. Then,
$$
\dfrac{\d }{\d t}\E(t) \leq -\mathbf{\Psi}_e(\E(t)) \qquad \qquad \forall t \geq 0.
$$
Moreover, there exists  $C >0$ such that
\begin{equation}\label{halfhaff}
\E(t) \leq C\left(1+t\right)^{-\frac{2}{1+\gamma}} \qquad \qquad \forall t \geq 0.
\end{equation}
\end{propo}
\begin{proof} Recall that the evolution of the temperature is given by
\begin{equation}\label{De}\dfrac{\d }{\d t}\E(t)=-\IRR f(t,v)f(t,\vb)\mathbf{\Psi}_e(|u|^2)\d v\d\vb, \qquad t \geq 0,\end{equation}
where $u=v-\vb.$ Since $\mathbf{\Psi}_e(|\cdot|^2)$ is convex according to Assumption \ref{HYP2} \textit{(2)} and $f(t,\vb)\d\vb$ is a probability measure over $\R^3$, Jensen's inequality implies
$$\IR f(t,\vb)\mathbf{\Psi}_e(|u|^2)\d\vb \geq \mathbf{\Psi}_e\left(\left|v-\IR \vb f(t,\vb)\d \vb\right|^2\right)=\mathbf{\Psi}_e(|v|^2) $$
where we used \eqref{firstmom}. Applying Jensen's inequality again we obtain
$$\IR f(t,v)\mathbf{\Psi}_e(|v|^2)\d v\geq \mathbf{\Psi}_e\left(\IR f(t,v)|v|^2\d v\right),$$
and therefore,
$$\dfrac{\d }{\d t}\E(t) \leq -\mathbf{\Psi}_e(\E(t)) \qquad \forall t \geq 0.$$
Note that $\mathbf{\Psi}_e(\cdot)$ is strictly increasing with $\lim_{x \to 0}\mathbf{\Psi}_e(x)=0$, this ensures that
$$\lim_{t \to \infty}\E(t)=0.$$
Moreover, according to Assumptions \ref{HYP2} \textit{(1)}, it is clear from \eqref{Psie} that
$$\mathbf{\Psi}_e(x) \simeq C_\gamma x^{\frac{3+\gamma}{2}} \quad \text{ for } \quad x \simeq 0,$$
where the constant can be taken as $C_\gamma=2\pi\alpha \int_0^1 y^{3+\gamma} b(1-2y^2)\d y<\infty.$  Since $\E(t) \to 0$, there exists $t_0 >0$ such that
$\mathbf{\Psi}_e(\E(t)) \geq \frac{1}{2}C_\gamma \E(t)^{\frac{3+\gamma}{2}}$ for all $t \geq t_0$
which implies that
$$\dfrac{\d }{\d t}\E(t) \leq -  \frac{C_\gamma}{2} \E(t)^{\frac{3+\gamma}{2}} \qquad \forall t \geq t_0.$$
This proves \eqref{halfhaff}
\end{proof}
\begin{exa} In the case of constant restitution coefficient $e(r)=e_0 \in (0,1)$ for any $r \geq 0$, for hard-spheres interactions, one has
$$\mathbf{\Psi}_e(x)=\dfrac{1-e^2_0}{8}x^{3/2}.$$
Thus, one recovers from \eqref{halfhaff} the decay of the temperature established from physical considerations (dimension analysis) in \cite{haff} and proved in \cite{MiMo2}, namely, $\E(t) \leq C(1+t)^{-2}$ for large $t$.
\end{exa}

\begin{exa} For the restitution coefficient $e(\cdot)$ associated to viscoelastic hard-spheres (see Example \ref{exa:visco}), one has $\gamma=1/5$, thus, the above estimate \eqref{halfhaff} leads to a decay of the temperature faster than $(1+t)^{-5/3}$ which is the one obtained in \cite{PoSc} (see also \cite{BrPo}) from physical considerations and dimensional analysis.\end{exa}

Notice that, since $\E(t) \to 0$ as $t \to \infty,$ it is possible to resume the arguments of \cite[Prop. 5.1]{MiMo} to  prove that the solution $f(t,v)$ to \eqref{cauch} converges to a Dirac mass as $t$ goes to infinity, namely
$$f(t,v) \underset{t \to \infty}{\longrightarrow} \delta_{v=0} \quad \text{ weakly $\ast$ in } \quad M^1(\R^3)$$
where $M^1(\R^3)$ denotes the space of normalized probability measures on $\R^3.$ We shall not investigate further on the question of long time asymptotic behavior of the distribution $f(t,v)$ but rather try to capture the very precise rate of convergence of the temperature to zero.\\

Using the Povzner-like estimate of Section \ref{sec:povzner} it is possible, from the decay in $\E(t)$, to deduce the decay of any moments of $f$.  Indeed, for any $t \geq 0$ and any $p \geq 1$ we define the \textit{p-moment} of $f$ as
\begin{equation}\label{defmp}
m_p(t):=\IR f(t,v)|v|^{2p}\d v.
\end{equation}
\begin{cor}\label{moments}
Let $f_0$ be a nonnegative velocity distribution satisfying \eqref{initial} and let $f(t,v)$ be the associated solution to the Cauchy problem \eqref{cauch} where the variable restitution coefficient satisfies Assumptions \ref{HYP2}. For any $p \geq 1$,  there exists $K_p >0$ such that
\begin{equation}\label{kp}
m_p(t) \leq K_p \left(1+t\right)^{-\frac{2p}{1+\gamma}} \qquad \forall t \geq 0.
\end{equation}
\end{cor}
\begin{proof} Set $u(t)=(1+t)^{-\frac{2}{1+\gamma}}$.  We prove that, for any $p \geq 1$, there exists $K_p >0$ such that $m_p(t) \leq K_p u^p(t)$ for any $t \geq 0$. Observe that using classical interpolation, it suffices to prove this for any $p$ such that $2p\in \mathbb{N}$. We argue by induction. It is clear from Proposition \ref{prop:cool} that estimate \eqref{kp} holds for $p=1$. Let $p>1$, with $2p \in \mathbb{N}$, be fixed and assume that for any integer $1 \leq j \leq p-1/2$ there exists $K_j >0$  such that $m_j(t) \leq K_j u^j(t)$ holds. According to Proposition \ref{povzner}
\begin{equation}\label{dmp}
\dfrac{\d }{\d t}m_p(t)=\IR \Q_{B,e}(f,f)(t,v)|v|^{2p}\d v\leq-(1-\kappa_{p})m_{p+1/2}(t)+\kappa_{p}\;S_{p}(t),
\end{equation}
where
\begin{equation*}
S_{p}(t)=\sum^{[\frac{p+1}{2}]}_{k=1}\left(
\begin{array}{c}
p\\k
\end{array}
\right)\left(m_{k+1/2}(t)\;m_{p-k}(t)+m_{k}(t)\;m_{p-k+1/2}(t)\right), \quad \forall t \geq 0.
\end{equation*}
For $p \geq 2$, the above expression $S_p(t)$ involves moments of order less than $p-1/2$.  The case $p=3/2$ is treated independently.\medskip

$\bullet$ \textit{Step 1 ($p=3/2$)}. In this case \eqref{dmp} reads
\begin{equation}\label{m3/2}
\dfrac{\d }{\d t}m_{3/2}(t) \leq -(1-\kappa_{3/2})m_{2}(t)+m_{3/2}(t)m_{1/2}(t)+\E^{2}(t) \qquad \forall t \geq 0.
\end{equation}
Let $K $ be a positive number to be chosen later and define
$$
U_{3/2}(t):=m_{3/2}(t)- Ku(t)^{3/2}.
$$
Using \eqref{m3/2} one has
\begin{multline*}
\dfrac{\d U_{3/2}}{\d t}(t) \leq -(1-\kappa_{3/2})m_{2}(t)+m_{3/2}(t)m_{1/2}(t)+\E^{2}(t)  +  \frac{3K}{1+\gamma}(1+t)^{-\frac{4+\gamma}{1+\gamma}}.
\end{multline*}
From Holder's inequality,
\begin{equation}\label{m3/2E}
m_{3/2}(t)\leq \sqrt{\E(t)} \sqrt{m_{2}(t)}  \quad \text{ and } \quad m_{1/2}(t)\leq \sqrt{\E(t)}  \qquad \forall t \geq 0
\end{equation}
hence,
\begin{equation*}
\dfrac{\d U_{3/2}}{\d t}(t) \leq  -(1-\kappa_{3/2}) \dfrac{m_{3/2}^2(t)}{\E(t)}
+\sqrt{\E(t)} m_{3/2}(t)+ \E^2 (t)+  \frac{3K}{ 1+\gamma }(1+t)^{-\frac{4+\gamma}{1+\gamma}}.
\end{equation*}
Since $\E(t) \leq C(1+t)^{-\frac{2}{1+\gamma}}$,  there exist $a ,b ,c >0$ such that
\begin{multline}\label{p3/2}\dfrac{\d U_{3/2}}{\d t}(t)\leq -a\, m_{3/2}^{2}(t)(1+t)^{\frac{2}{1+\gamma}} + b\,(1+t)^{-\frac{4}{1+\gamma}}\\
+  c\,(1+t)^{-\frac{1}{1+\gamma}}m_{3/2}(t)+\frac{3}{1+\gamma} K(1+t)^{-\frac{4+\gamma}{1+\gamma}}\qquad \forall t >0.\end{multline}
Inequality \eqref{p3/2} implies the result for the case $p=3/2$ provided $K$ is large enough.  Indeed, choose $K$ so that $m_{3/2}(0)<K u^{3/2}(0)=K$.  Then, by time-continuity of the moments, the result follows at least for some finite time.  Assume that there exists a time $t_{\star}>0$ such that $m_{3/2}(t_{\star})=K u^{3/2}(t_{\star})=K(1+t_\star)^{-\frac{3}{1+\gamma}}$, then \eqref{p3/2} implies
$$
\dfrac{\d U_{3/2}}{\d t}(t_\star)\leq \left(-aK^2+b+cK+\frac{3}{1+\gamma}K\right)(1+t_\star)^{-\frac{4}{1+\gamma}} <0
$$
whenever $K$ is large enough. Thus, \eqref{kp} holds for $p=3/2$ choosing $K_{3/2}:=K$.\medskip

$\bullet$ \textit{Step 2 ($p \geq 2$).}  The induction hypothesis implies that there exists a constant $C_p>0$ such that
$$
S_p(t) \leq C_p\, u(t)^{p+1/2}, \qquad \forall t \geq 0
$$
where $C_p$ can be taken as
$$
C_p=\sum^{[\frac{p+1}{2}]}_{k=1}\left(
\begin{array}{c}
p\\k
\end{array}
\right)\left(K_{k+1/2}\;K_{p-k}+K_{k}\;K_{p-k+1/2}\right).
$$
Furthermore, according to Jensen's inequality $m_{p+1/2}(t)\geq m_{p}^{1+1/2p}(t),$ for any $t \geq 0.$ Thus, from \eqref{dmp}, we conclude that
\begin{equation*}
\dfrac{\d }{\d t}m_p(t) \leq -(1-\kappa_{p})m_{p}^{1+1/2p}(t)+ \kappa_p\;C_p \,u(t)^{p+1/2} \qquad \forall t \geq 0.
\end{equation*}
Arguing as in \textit{Step 1}, for some $K >0 $ to be chosen later, we define
$$
U_p(t):=m_p(t)- Ku(t)^{p}.$$
In this way,
$$
\dfrac{\d }{\d t}U_p(t) \leq -(1-\kappa_{p})m_{p}^{1+1/2p}(t)+ \kappa_p\;C_p \,u(t)^{p+1/2} +\dfrac{2pK}{1+\gamma}(1+t)^{-\frac{2p+1}{1+\gamma}} \qquad \forall t \geq 0.
$$
Then, if $K$ is such that $U_p(0) < 0$, the result holds at least for some finite time. For any $t_\star >0$ such that $U_p(t_\star)=0$, one notices then that
\begin{equation*}
\dfrac{\d }{\d t}U_p(t_\star)  \leq \left(-(1-\kappa_p)K^{1+\frac{1}{2p}}+\kappa_p\,C_p+\dfrac{2pK}{1+\gamma}\right)\,(1+t_\star)^{-\frac{2p+1}{1+\gamma}}<0
\end{equation*}
provided $K$ is large enough. This proves \eqref{kp} for any $p \geq 1$.
\end{proof}

\subsection{Lower bound for $\E(t)$: preliminary considerations} The next goal is to complete the proof of Haff's law by showing that the cooling rate \eqref{halfhaff} is optimal under Assumptions \ref{HYP2}.  Thus, we have to show that there exists $C>0$ such that
$$
\E(t) \geq C(1+t)^{-\frac{2}{1+\gamma}}\qquad \forall t \geq 0.
$$
First, we prove the following result that simplifies our endeavor.
\begin{theo}\label{momlam}
Assume a non constant ($\gamma >0$) restitution coefficient $e(\cdot)$ satisfying Assumptions \ref{HYP2}. If there exist $C_0 >0$ and $\lambda >0$ such that
\begin{equation}\label{lam}
\E(t) \geq  C_0\,(1+t)^{-\lambda} \qquad \forall t \geq 0,
\end{equation}
then there exists $C_p >0$ such that
\begin{equation}\label{kp2}
m_{p}(t)\leq C_{p}\,\E^{p}(t)  \qquad \text{ for any } t \geq 0\ \ \mbox{and}\ \ p \geq 1.
\end{equation}
As a consequence, there exists $C >0$ such that
\begin{equation}\label{converseE}
\E(t) \geq  C\,(1+t)^{-\frac{2}{1+\gamma}}  \qquad \forall t \geq 0.
\end{equation}
\end{theo}
\begin{proof} According to Assumption \ref{HYP2} \textit{(1)}
$$\mathbf{\Psi}_e(x) \simeq C_\gamma x^{\frac{3+\gamma}{2}}\ \ \mbox{for}\ \ x \simeq 0.
$$
In addition, Assumption \ref{HYP2} \textit{(2)} implies that there exists $C_b >0$  such that
$$
\mathbf{\Psi}_e(x) \simeq C_b x^{3/2}\ \ \mbox{for large}\ x,
$$
where the constant can be taken as $C_b =2\pi(1-e^2_0)\int_0^1 b(1-2z^2)z^3\d z.$ Thus, there exists another constant $C >0$ such that
\begin{equation}\label{psieC}
\mathbf{\Psi}_e(x) \leq C x^{\frac{3+\gamma}{2}} \qquad \forall x >0.
\end{equation}
Then, from \eqref{De} one deduces that for any $\varepsilon >0$ and $p > \frac{3+\gamma}{2}$
\begin{equation*}
\begin{split}
-\dfrac{\d}{\d t}\E(t) &\leq C \left(\varepsilon^{\frac{\gamma}{2}}m_{3/2}(t)+\frac{1}{\varepsilon^{p-\frac{3+\gamma}{2}}}m_{p}(t)\right) \\
&\leq C\left(\varepsilon^{\frac{\gamma}{2}}m_{3/2}(t)+\frac{C_p}{\varepsilon^{p-\frac{3+\gamma}{2}}}(1+t)^{-\frac{2p}{1+\gamma}}\right) \qquad \forall t \geq 0.
\end{split}
\end{equation*}
where we have used Corollary \ref{moments} for the second inequality. In particular, using \eqref{lam} and the fact that $\E(t)$ is a non increasing function, one can choose $p$ sufficiently large so that
$$-\dfrac{\d}{\d t}\E(t)\leq C\left(\varepsilon^{\frac{\gamma}{2}}m_{3/2}(t)+\frac{\tilde{C}_p}{\varepsilon^{p-\frac{3+\gamma}{2}}}\E(t)^{\frac{3}{2}}\right)$$
for some positive constant $\tilde{C}_p$. In other words, for any $\delta >0$ there exists $C_\delta >0$ such that
\begin{equation}\label{obs}
-\dfrac{\d}{\d t}\E(t)\leq \delta m_{3/2}(t) + C_\delta \E(t)^{3/2} \qquad \forall t \geq 0.
\end{equation}
With this preliminary observation, the proof of \eqref{kp2} is a direct adaptation of that of Corollary \ref{moments}. Here again, by simple interpolation, it is enough to prove the result for any $p$ such that $2p \in \mathbb{N}$ and argue using induction. The result is clearly true for $p=1$ with $C_1=1$. For $p=3/2$, let $K >0$ be a constant chosen later and define
$$u_{3/2}(t)=m_{3/2}(t)-K\E(t)^{3/2}.$$
Thus, from \eqref{m3/2}
$$
\dfrac{\d }{\d t}u_{3/2}(t)\leq -(1-\kappa_{3/2})m_{2}(t)+m_{3/2}(t)m_{1/2}(t)+\E^{2}(t)-\frac{3}{2}K\sqrt{\E(t)}\frac{\d}{\d t}\E(t).
$$
Using \eqref{m3/2E} one deduces from \eqref{obs} that, for any $\delta >0$ there exists $C_\delta >0$ such that
\begin{multline*}
\dfrac{\d }{\d t}u_{3/2}(t)\leq -(1-\kappa_{3/2})\frac{m_{3/2}^2(t)}{\E(t)}  + \left(1+ \frac{3}{2}K\delta\right)m_{3/2}(t)\sqrt{\E(t)}
+\left(1+\frac{3}{2}K C_\delta\right)\E^2(t).
\end{multline*}
Fix $\delta=\frac{1-\kappa_{3/2}}{3}$ and choose $K >0$ such that $u_{3/2}(0) <0$. If $t_{\star}>0$ is such that $u_{3/2}(t_\star)=0$, then the following holds
$$\dfrac{\d }{\d t}u_{3/2}(t_{\star})\leq
\left(-\frac{1-\kappa_{3/2}}{2}K^{2}+ \left(K +1+\frac{3}{2}KC_\delta\right)\right)\E(t_{\star})^{2}<0,$$
provided $K$ is sufficiently large. This proves \eqref{kp2} for $p=3/2$ with $C_{3/2}:=K$.  The case $p\geq2$ follows in the same lines of the proof of Corollary \ref{moments} interchanging the roles of $\E(t)$ and $u(t)$.\medskip

To conclude the proof, observe that according to \eqref{psieC} and \eqref{De}, there exists $C >0$ such that
$$
-\frac{\d}{\d t}\E(t) \leq C m_{\frac{3+\gamma}{2}}(t)\qquad \forall t\geq 0.
$$
Then, applying \eqref{kp2} with $p=\frac{3+\gamma}{2}$, one deduces that there is $C_\gamma >0$ such that
$$
-\dfrac{\d}{\d t}\E(t) \leq C_\gamma \E(t)^{\frac{3+\gamma}{2}} \qquad \forall t \geq 0.
$$
A simple integration of this inequality yields \eqref{converseE}.
\end{proof}
\begin{nb} For constant restitution coefficient, $e=e_0$, since $\gamma=0$,  \eqref{obs} does not hold anymore.  However, for some $C_e >0$ we have
$$-\dfrac{\d}{\d t}\E(t) \leq C_e m_{3/2}(t) \qquad \forall t \geq 0.$$
Assuming that $e_0 \simeq 1$ (quasi-elastic regime) the constant $C_e$ is small, thus, the argument above can be reproduced to prove that the conclusion of Proposition \ref{momlam} still holds. Recall that for $\gamma=0$ the second part of Haff's law \eqref{converseE} has been proved in  \cite[Theorem 1.2]{MiMo2}.\end{nb}

In order to prove that \eqref{lam} is satisfied for some $\lambda >0$, we will need precise $L^p$ estimates, following the spirit of \cite{MiMo2}, for the rescaled function $g$ given in Section \ref{sec:Self}.  The idea to craft the correct time-scaling functions $\tau(\cdot)$ and $V(\cdot)$ is to choose them such that the corresponding temperature of $g$ is bounded away from zero.  Indeed, for any $\tau >0$, define
$$
\mathbf{\Theta}(\tau):=\IR g(\tau,w)|w|^2 \d w.
$$
Since,
\begin{equation}\label{Eg}
\E(t)=V(t)^{-2}\mathbf{\Theta}(\tau(t)) \qquad \forall t \geq 0,
\end{equation}
we choose
\begin{equation}\label{V(t)}
V(t)=(1+t)^{\frac{1}{\gamma+1}} \qquad \forall  t \geq 0.
\end{equation}
In this way, \eqref{converseE} is equivalent to $\mathbf{\Theta}(\tau(t)) \geq C$ for any $t \geq 0$. Notice that \eqref{halfhaff} immediately translates into
\begin{equation}\label{halfgtau}\sup_{t >0}\mathbf{\Theta}(\tau(t)) < \infty.\end{equation}
Moreover, for simplicity we pick $\tau(t)$ such that $\dot{\tau}(t)V(t)=1$, therefore for $\gamma>0$,
\begin{equation}\label{exptau}
\tau(t)=\int_0^t \dfrac{\d s}{V(s)}=\dfrac{\gamma+1}{\gamma}\left((1+t)^{\frac{\gamma}{1+\gamma}}-1\right)
\end{equation}
which is an acceptable time-scaling function.  Thus, the rescaled solution $g(\tau,w)$ satisfies \eqref{eqgt} with $\lambda(\tau)=1$,
\begin{equation}\label{xitau}
\xi(\tau)=\dfrac{1}{\gamma \tau +  (1+\gamma)} \quad \text{and} \quad  \widetilde{e}_\tau(r)=e\left(r\left(1+\frac{\gamma}{\gamma+1}\tau\right)^{-1/\gamma}\right).
\end{equation}
If $\gamma=0$ the restitution coefficient is constant \cite{MiMo2}, in particular $\widetilde{e}_\tau=e$, and the rescale reads $V(t)=1+t$ and $\tau(t)=\ln(1+t)$. In such a case, $\xi(\tau)\equiv1$.\medskip

To complete the proof of Haff's law,  one has to perform a careful study of the properties of the collision operator $\Q_e$ in Sobolev or $L^p$ spaces $1 < p \leq \infty.$

\section{Regularity properties of  the collision operator}

In this section the regularity properties studied originally for the elastic case in \cite{lions,MoVi,Wennberg} and later for the constant restitution coefficient in \cite{MiMo2} are generalized to cover variable restitution coefficients depending on the impact velocity.  The path that we follow closely follows \cite{MoVi}.

\subsection{Carleman representation}\label{sub:carle}
We establish here a technical representation of the gain term $\Q^+_{B,e}$ which is reminiscent of the classical Carleman representation in the elastic case.  More precisely, let $B(u,\sigma)$ be a collision kernel of the form
$$
B(u,\sigma)=\Phi(|u|)b(\widehat{u} \cdot \sigma)
$$
where $\Phi(\cdot) \geq 0$ and $b(\cdot) \geq 0$ satisfies \eqref{normalization}. For any $\psi=\psi(v)$, define the following linear operators
\begin{equation}\label{S_+-}
\mathcal{S}_\pm(\psi)(u)=\IS \psi(u^\pm)b(\widehat{u}\cdot \sigma)\d\sigma, \qquad \forall u \in \mathbb{R}^3,
\end{equation}
the symbolos $u{^-}$ and $u^{+}$ are defined by $$u^-:=\beta\left(|u\;|\sqrt{\frac{1-\widehat{u}\cdot\sigma}{2}}\right)\;\frac{u-|u|\;\sigma}{2}, \quad \text{ and } \quad u^+:=u-u^-.$$
\begin{lemme}\label{gammaB}
For any continuous functions $\psi$ and $\varphi$,
$$
\IR \varphi(u)\mathcal{S}_-(\psi)(u)\Phi(|u|) \d u=\IR \psi(x)\Gamma_B (\varphi) (x)\d x
$$
where the linear operator $\Gamma_B$ is given by
\begin{multline}\label{expgammaB}
\Gamma_B(\varphi)(x)=\int_{\omega^\perp} \mathcal{B}(z+\alpha(r)\omega,\alpha(r))\varphi(\alpha(r)\omega+z)\d\pi_z,\\
 \qquad x=r\omega, \:r\geq 0,\:\omega \in \mathbb{S}^2.
 \end{multline}
Here $\d\pi_{z}$ is the Lebesgue measure in the hyperplane $\omega^\perp$ perpendicular to $\omega$ and $\alpha(\cdot)$ is the inverse of the mapping $s \mapsto s\beta(s)$.  Moreover,
\begin{equation}\label{calB}
\mathcal{B}(z,\varrho)= \frac{8\Phi(|z|)}{|z|(\varrho \beta(\varrho))^2}b\left(1-2\frac{\varrho^2}{|z|^2}\right)\dfrac{\varrho}{1+\vartheta'(\varrho)}, \qquad \varrho \geq 0, \quad z \in \R^3
\end{equation}
with $\vartheta(\cdot)$ defined in Assumption \ref{HYP} (2) and $\vartheta'(\cdot)$ denoting its derivative.
\end{lemme}
\begin{proof} For simplicity assume that $\Phi\equiv 1$.  Define
$$I:=\IR \varphi(u)\mathcal{S}_-(\psi)(u)\d u=\IR \varphi(u)\d u \IS \psi(u^-)b(\widehat{u}\cdot \sigma)\d\sigma.
$$
For fixed $u \in \R^3$, we perform the integration over $\mathbb{S}^2$ using the formula
$$
\IS F\left(\dfrac{u-|u|\sigma}{2}\right)\d\sigma=\dfrac{4}{|u|}\IR \delta(|x|^2- x \cdot u ) F(x)\d x
$$
valid for any given function $F$. Then,
$$
I=4\IRR \varphi(u)|u|^{-1} \delta(|x|^2- x \cdot u )\psi\big(x\beta(|x|)\big)b\left(1-2\frac{|x|^2}{|u|^2}\right)\d x\d u.
$$
Setting now $u=z+x$ we get
$$I=4\IRR \varphi(x+z)|x+z|^{-1} \delta( x \cdot z )\psi\big(x\beta(|x|)\big)b\left(1-2\frac{|x|^2}{|x+z|^2}\right)\d z\d x.$$
Keeping $x$ fixed, we remove the Dirac mass using to the identity
$$\IR F(z)\delta(x \cdot z)\d z =\frac{1}{|x|} \int_{x^\perp} F(z)\d\pi_z,$$
which leads to
$$I=4\IR \psi\big(x\beta(|x|)\big)\frac{\d x}{|x|}\int_{x^\perp} \dfrac{\varphi(x+z)}{|x+z|}b\left(1-2\frac{|x|^2}{|x+z|^2}\right)\d \pi_z.$$
Perform the $x$--integral using polar coordinates $x=\varrho \,\omega$ and the change of variables $r=\varrho\, \beta(\varrho)$.  Recall that $\alpha(r)$ is the inverse of such mapping, furthermore, notice that $\d r=\frac{1}{2}(1+\vartheta'(\varrho))\d\varrho$.  This yields
\begin{multline*}
I=8\int_0^\infty \dfrac{\alpha(r)\d r}{1+\vartheta'(\alpha(r))}\IS \psi(r\omega)\d\omega\int_{\omega^\perp}\dfrac{\varphi(z+\alpha(r)\omega)}{|z+\alpha(r)\omega|}b\left(1-2\frac{\alpha(r)^2}{|z+\alpha(r)\omega|^2}\right)\d\pi_z.
\end{multline*}
Turning back to cartesian coordinates $x=r\omega$ we obtain the desired expression  $$I=\IR \psi(x)\Gamma_B(\varphi)(x)\d x,$$
with $\Gamma_B$ given by \eqref{expgammaB}.
\end{proof}
The above result leads to a Carleman-like expression for $\Q^+_{B,e}$:
\begin{cor}[\textbf{Carleman representation}]\label{carle} Let $e(\cdot)$ be a restitution coefficient satisfying Assumptions \ref{HYP} and let
$$
B(u,\sigma)=\Phi(|u|)b(\widehat{u} \cdot \sigma)
$$
be a collision kernel satisfying \eqref{normalization}. Then, for any velocity distribution functions $f,g$ one has
$$\Q_{B,e}^+(f,g)(v)=\IR f(z)\left[\left(t_z  \circ \Gamma_B \circ t_z\right)g\right] (v)\d z $$
where $[t_v \psi](x)=\psi(v-x)$ for any $v,x \in \R^3$ and test-function $\psi$.
\end{cor}
\begin{proof} The proof readily follows from the Lemma \ref{gammaB} and the identity
\begin{equation}\label{weaS-}
\IRR \Q_{B,e}^+(f,g)(v) \psi(v)\d v=\dfrac{1}{2}\IRR f(v)g(v-u)\Phi(|u|)\mathcal{S}_-(t_v \psi)(u)\d v\d u
\end{equation}
valid for any test-function $\psi$.\end{proof}

\subsection{Convolution-like estimates for $\Q_{B,e}^+$} General convolution-like estimates are obtained in \cite[Theorem 1]{AloCar} for non-constant restitution coefficient. Such estimates are given in $L^p_\eta$ with $\eta \geq 0$ and, for the applications we have in mind, we need to extend some of them to $\eta \leq 0$. This can be done using the method developed in \cite{MoVi} (see also \cite{GaPaVi}) together with the estimates of \cite{AloCar}.
\footnote{Notice that the constants ${\gamma}(\eta,p,b)$ and $\widetilde{\gamma}(\eta,p,b)$ given by \eqref{Cetab} and \eqref{Cetabtilde} are not finite for arbitrary angular kernel $b$.  It is implicitly assumed that the Theorem applies for the range of parameters leading to finite constants (see also Remark \ref{nbgamma}).}
\begin{theo} \label{alo}
Assume that the collision kernel $B(u,\sigma)=\Phi(|u|)b(\widehat{u} \cdot \sigma)$ satisfies \eqref{normalization} and  $\Phi(\cdot) \in L^\infty_{-k}$ for some $k \in \R$. In addition, assume that $e(\cdot)$ fulfills Assumption \ref{HYP}. Then, for any $1 \leq p \leq \infty$  and $\eta \in \R$, there exists $\mathbf{C}_{\eta,p,k}(B)>0$ such that
$$
\left\|\Q^+_{B,e}(f,g)\right\|_{L^p_\eta} \leq \mathbf{C}_{\eta,p,k}(B) \,\|f\|_{L^1_{|\eta +k|+|\eta|}} \,\|g\|_{L^p_{\eta +k}}
$$
where the constant $\mathbf{C}_{\eta,p,k}(B)$ is given by:
\begin{equation}\label{C_k(B)}
\mathbf{C}_{\eta,p,k}(B)=c_{k,\eta,p}\;\gamma(\eta,p,b)\left\|\Phi\right\|_{L^\infty_{-k}}
\end{equation}
with a constant $c_{k,\eta,p} >0$ depending only on $k,\eta$ and $p$.  Furthermore, the dependence on the angular kernel is given by
\begin{equation}\label{Cetab}
\gamma(\eta,p,b)=\int_{-1}^1 \left(\frac{1-s}{2}\right)^{-\frac{3+\eta_+}{2p'}}b(s)\d s,
\end{equation}
where $1/p+1/p'=1$ and $\eta_+$ is the positive part of $\eta$. Similarly, there exists $\widetilde{\mathbf{C}}_{\eta,p,k}(B)>0$ such that
$$
\left\|\Q^+_{B,e}(f,g)\right\|_{L^p_\eta} \leq \widetilde{\mathbf{C}}_{\eta,p,k}(B) \,\|g\|_{L^1_{|\eta +k|+|\eta|}} \,\|f\|_{L^p_{\eta +k}}
$$
where the constant $\widetilde{\mathbf{C}}_{\eta,p,k}(B)$ is given by
\begin{equation}\label{C_k(B)tilde}
\widetilde{\mathbf{C}}_{\eta,p,k}(B)=\widetilde{c}_{k,\eta,p} \;\widetilde{\gamma}(\eta,p,b)\left\|\Phi\right\|_{L^\infty_{-k}}
\end{equation}
for some constant $\widetilde{c}_{k,\eta,p} >0$ depending only on $k,\eta$ and $p$.  The dependence on the angular kernel is given by
\begin{equation}\label{Cetabtilde}
\widetilde{\gamma}(\eta,p,b)=\int_{-1}^1 \left(\frac{1+s}{2}+\left(1-\beta_0\right)^2\frac{1-s}{2}\right)^{-\frac{3+\eta_+}{2p'}}b(s)\d s
\end{equation}
where $1/p+1/p'=1$ and $\beta_0=\beta(0)=\frac{1+e(0)}{2}.$
\end{theo}
\begin{proof} Fix $1\leq p \leq \infty$ and $\eta \in \R$ and use the convention $1/p'+1/p=1$. By duality,
 \begin{equation*}
 \left\| \Q^+_{B,e}(f,g)\right\|_{L^p _\eta} =
 \sup \left\{\left| \IR\Q^+_{B,e}(f,g)(v)\psi(v)\d v \right|\ \ ; \ \left\|\psi\right\|_{L^{p'} _{-\eta}} \leq 1 \right\}.
 \end{equation*}
Using \eqref{weaS-},
\begin{equation*}
\IR \Q^+_{B,e}(f,g)(v)\psi(v)\d v=\IRR f(v)g(v-u)\mathcal{T}_-(t_v \psi)(u)\d v \d u \end{equation*}
with
$$
\mathcal{T}_-(\psi)(u)=\Phi(|u|)\mathcal{S}_-(\psi)(u),\ \ \mbox{and} \ \ \ t_v \psi(x)=\psi(v-x),
$$
with $\mathcal{S}_-$ defined in equation \eqref{S_+-}.  With the notation of \cite{AloCar}, one recognizes that $\mathcal{S}_-(h)=\mathcal{P}(h,1)$, thus, applying \cite[Theorem 5]{AloCar} with $q=\infty$ and $\alpha=-\eta$,
$$
\|\mathcal{S}_-(h)\|_{L^{p'}_{-\eta}} \leq \gamma(\eta,p,b)\|h\|_{L^{p'}_{-\eta}}
$$
with $\gamma(\eta,p,b)$ given by \eqref{Cetab}.  Notice that, with respect to \cite{AloCar}, we used the weight $\langle v \rangle^\eta$ instead of $|v|^\eta$, this is the reason to introduce $\eta_+$ in our definition of $\gamma(\eta,p,b)$. As a consequence,
\begin{equation}\label{matT}\|\mathcal{T}_-(h)\|_{L^{p'}_{-\eta-k}} \leq \gamma(\eta,p,b)\|\Phi\|_{L^\infty_{-k}} \|h\|_{L^{p'}_{-\eta}}.\end{equation}
Now,
\begin{equation*}\begin{split}
\left| \IR \Q^+_{B,e}(f,g)\psi\d v \right| &\leq \IR |f(v)|\d v \left(\IR |g(u)| \left[(t_v \circ \mathcal{T}_- \circ t_v)\psi\right](u)\d u\right)\\
&\leq \|g\|_{L^p_{\eta+k}} \IR |f(v)| \left\|(t_v \circ \mathcal{T}_- \circ t_v)\psi\right\|_{L^{p'}_{-k-\eta}}\d v.\end{split}
\end{equation*}
Using the inequality $\|t_v h\|_{L^{p'}_s} \leq 2^{|s|/2} \langle v \rangle^{|s|} \|h\|_{L^{p'}_s}$ for any $s \in \R$ and $v$,
\begin{multline*}
\left| \IR \Q^+_{B,e}(f,g)\psi\d v \right|  \leq 2^{|\eta+k|/2}\|g\|_{L^p_{\eta+k}} \IR |f(v)|\langle v \rangle^{|\eta+k|}\left\|( \mathcal{T}_- \circ t_v)\psi\right\|_{L^{p'}_{-k-\eta}}\d v\\
 \leq 2^{|\eta+k|/2}\gamma(\eta,p,b)\|\Phi\|_{L^\infty_{-k}}\|g\|_{L^p_{\eta+k}} \IR |f(v)|\langle v \rangle^{|\eta+k|}\left\|  t_v \psi\right\|_{L^{p'}_{ -\eta}}\d v\\
 \leq 2^{|\eta+k|+|\eta|/2}\gamma(\eta,p,b)\|\Phi\|_{L^\infty_{-k}}\|g\|_{L^p_{\eta+k}} \IR |f(v)|\langle v \rangle^{|\eta+k|+|\eta|}\left\|\psi\right\|_{L^{p'}_{ -\eta}}\d v\end{multline*}
which proves the first part of the Theorem. To prove the second part, observe that
\begin{equation*}
\IR \Q^+_{B,e}(f,g)(v)\psi(v)\d v=\IRR f(v-u)g(v)\mathcal{T}_+(t_v \psi)(u)\d v \d u, \end{equation*}
where $\mathcal{T}_+(\psi)(u)=\Phi(|u|)\mathcal{S}_+(\psi)(u)$ and $\mathcal{S}_+$ defined in \eqref{S_+-}. Using the notation of \cite{AloCar} we identify $\mathcal{S}_+(h)=\mathcal{P}(1,h)$.  Thus, applying \cite[Theorem 5]{AloCar} with $p=\infty$ and $\alpha=-\eta$,
$$
\|\mathcal{S}_-(h)\|_{L^{p'}_{-\eta}} \leq \widetilde{\gamma}(\eta,p,b)\|h\|_{L^{p'}_{-\eta}}
$$
where $\widetilde{\gamma}(\eta,p,b)$ given by \eqref{Cetabtilde}. One concludes as above, interchanging the roles of $f$ and $g$.
\end{proof}
\begin{nb}\label{nbgamma} The constants $\gamma(\eta,p,b)$ and $\widetilde{\gamma}(\eta,p,b)$ are not finite for arbitrary $b(\cdot)$ because of the possible singularity at $s=\pm1$. However, if one assumes, as in \cite{MiMo2}, that the angular kernel $b(\cdot)$ vanishes in the vicinity of $s=1$ then $\gamma(\eta,p,b) < \infty$ for any $1 \leq p \leq \infty$ and $\eta \in \R$.  This is an additional difficulty of the inelastic regime that is overcome in the elastic case using symmetry, i.e., defining $b$ in half the domain.  The careful reader will also notice that the constants given in the theorem are independent of $e(\cdot)$ except for $\tilde{\gamma}(\eta,p,b)$ which depends only on the value $e(0)$. Finally, we mention that, for hard-sphere interactions, i.e. $b\equiv \frac{1}{4\pi}$, one has $\gamma(\eta,p,b) < \infty \:\:\Longleftrightarrow\:\: \widetilde{\gamma}(\eta,p,b) < \infty \:\:\Longleftrightarrow\:\: 1 \leq p < \frac{3+\eta_+}{1+\eta_+}.$
\end{nb}
\subsection{Sobolev regularity for smooth collision kernel.}
For this section we assume $\Phi(\cdot)$ and $b(\cdot)$ smooth and compactly supported
\begin{equation}\label{smoothphi}
\Phi \in \mathcal{C}_0^\infty (\mathbb{R}^3 \setminus \{0\} ), \qquad b \in \mathcal{C}_0^\infty(-1,1).
\end{equation}
Denote by $\Q_{B,e}$ the associated collision operator defined by \eqref{Ie3B}.
\begin{lemme}\label{Sob}
Assume that $e(\cdot)$ satisfies Assumptions \ref{HYP} with $e(\cdot) \in \mathcal{C}^m(0,\infty)$ for some integer $m \in \mathbb{N}.$    Then, under assumption \eqref{smoothphi} on the collision kernel, for any $0 \leq s \leq m$, there exists $C=C(s,B,e)$ such that
$$
\left\|\Gamma_B(f)\right\|_{H^{s+1}} \leq C(s,B,e)\,\left\|f\right\|_{H^s},\qquad \forall f \in H^s
$$
where $\Gamma_B$ is the operator defined in Lemma \ref{gammaB}.  The constant $C(s,B,e)$ depends only on $s$, on the collision kernel $B$ and the restitution coefficient $e(\cdot)$.  More precisely, $C(s,B,e)$ depends on $e(\cdot)$ through the $L^\infty$ norm  of the derivatives $D^k e(\cdot)$ ($k=1,\ldots,m$) over some compact interval bounded away from zero depending only on $B$.
\end{lemme}
We postpone the proof of Lemma \ref{Sob} and first prove its important consequence.
\begin{theo}\label{theo:smooth}
Let $B(u,\sigma)=\Phi(|u|)b(\widehat{u} \cdot \sigma)$ be a collision kernel satisfying \eqref{smoothphi} and $e(\cdot)$ satisfying Assumption \ref{HYP}.  In addition, assume that $e(\cdot) \in \mathcal{C}^m(0,\infty)$ for some integer $m \in \mathbb{N}$.  Then, for any $0 \leq s \leq m$,
$$
\left\|\Q_{B,e}^+(f,g)\right\|_{H^{s+1}} \leq C(s,B,e)\,\|g\|_{H^s}\,\|f\|_{L^1}
$$
with constant $C(s,B,e)$ given in Lemma \ref{Sob}.
\end{theo}
\begin{proof}
Let $\F\left[\Q_{B,e}^+(f,g)\right](\xi)$ denote the Fourier transform of $\Q_{B,e}^+(f,g)$. According to Corollary \ref{carle},
$$\F\left[\Q_{B,e}^+(f,g)\right](\xi)=\IR f(v) \F \left[\left(t_v  \circ \Gamma_B \circ t_v\right)g\right](\xi) \d v.$$
To simplify notation set $G(v,\xi)=\F \left[\left(t_v  \circ \Gamma_B \circ t_v\right)g\right](\xi)$. Thus,
\begin{equation}\label{four}\begin{split}
\left\|\Q_{B,e}^+(f,g)\right\|_{H^{s+1}}^2&=\IR \left|\F\left[\Q_{B,e}^+(f,g)\right](\xi)\right|^2 \langle \xi \rangle^{2(s+1)} \d\xi\\
&=\IR \langle \xi \rangle^{2(s+1)} \left|\IR f(v)G(v,\xi)\d v\right|^2\d\xi\\
&\leq \|f\|_{L^1} \IRR |f(v)|\,|G(v,\xi)|^2 \langle \xi \rangle^{2(s+1)} \d\xi \d v.
\end{split}
\end{equation}
Since $G(v,\xi)=\F \left[\left(t_v  \circ \Gamma_B \circ t_v\right)g\right](\xi)$,
$$\IR |G(v,\xi)|^2 \langle \xi \rangle^{2(s+1)} \d\xi=\left\| \left(t_v  \circ \Gamma_B \circ t_v\right)g\right\|^2_{H^{s+1}}\leq C(s,B,e)^2\,\left\|g\right\|_{H^s}^2.$$
For this inequality we used Lemma \ref{Sob} and the fact that the translation operator $t_v$ has norm one in any Sobolev space.  Hence, estimate \eqref{four} yields the desired  estimate.
\end{proof}
\begin{proof}[Proof of Lemma \ref{Sob}]
The proof of the regularity property of $\Gamma_B$ can be obtained following the lines of the one for the elastic Boltzmann operator \cite{MoVi}.  Indeed, note that
\begin{align*}
\widetilde{\Gamma_B}(f)(r,\omega):&=\Gamma_B(f)(\alpha^{-1}(r),\omega)=\Gamma_B(f)(r\beta(r),\omega)\\
&=\int_{\omega^\perp} \mathcal{B}(z+r\omega,r)\varphi(r\omega+z)\d\pi_z.
\end{align*}
Assumption \eqref{smoothphi} implies that there exists $\delta >0$ such that $b(x)=0$ for $|x \pm 1| \leq \delta$ and  $\left\{|z|\,;\, z \in \mathrm{Supp}(\Phi)\right\} \subset (a,M)$ for some positive constants $0<a<M$. Then, by virtue of \eqref{calB}, $\mathcal{B}(z+r\omega,r)=0$ for any $r >0$, $\omega \in \mathbb{S}^2$ and $z \in \omega^\bot$ provided that $|z|^2 > \frac{2-\delta}{\delta}r^2$.  For $|z|^2 \leq \frac{2-\delta}{\delta}r^2$, one has $|z+r\omega|^2 \leq 2r^2/\delta$, thus, $\mathcal{B}(z+r\omega,r)=0$ if $r < \sqrt{\delta a^2/2}$.  Putting these together we conclude that
\begin{equation}\label{inI}
\mathcal{B}(z+r\omega,r)=0 \qquad \forall r \notin I:=\left(\sqrt{\delta a^2/2},M \right), \: \omega \in \mathbb{S}^2  \text{ and any }  z \bot \omega.
\end{equation}
In particular, $\widetilde{\Gamma_B}(f)(r,\omega)=0$ for any $r \notin I$ independently of $f$. Define
$$\mathcal{B}_0(z,\varrho)
:=\dfrac{1+\vartheta'(\varrho)}{\varrho}\beta^2(\varrho)\mathcal{B}(z,\varrho)=\dfrac{\Phi(|z|)b\left(1-2\frac{\varrho^2}{|z|^2}\right)}{|z|\varrho^2} $$
and denote $\widetilde{\Gamma_0}(f)$ the associated operator,
$$
\widetilde{\Gamma_0}(f)(r,\omega):=\int_{\omega^\perp} \mathcal{B}_0(z+r\omega,r)\varphi(r\omega+z)\d\pi_z.
$$
Then, $\mathcal{B}_0$ does not depend on the restitution coefficient $e(\cdot)$ and $\widetilde{\Gamma_0}$ is exactly of the form of the operator $T$ studied in \cite[Theorem 3.1]{MoVi}. Therefore, arguing as in \textit{op. cit.}, for any $s  \geq 0$, there is an explicit constant $C_0=C_0(s,\Phi,b)$ such that
\begin{equation}\label{coVil}\left\|\widetilde{\Gamma_0}(f)\right\|_{H^{s+1}} \leq C_0(s,\Phi,b)\,\left\|f\right\|_{H^s},\qquad \forall f \in H^s.\end{equation}
Setting
\begin{equation}\label{Ge}
G_e(\varrho)=\dfrac{\varrho}{\left(1+\vartheta'(\varrho)\right)\beta^2(\varrho)} \qquad \forall \varrho \geq 0,
\end{equation}
one observes that $G_e$ is a $\mathcal{C}^m$ function over $I$ whose derivatives $D^k G_e$ are bounded over $I$ for any $k \leq m$ and
\begin{equation*}
\widetilde{\Gamma_B}(f)(r,\omega)=G_e(r)\chi_I(r) \widetilde{\Gamma_0}(f)(r,\omega).
\end{equation*}
Here $\chi_I$ is the characteristic function of $I=\left(\sqrt{\delta a^2/2},M \right)$ (see Eq. \eqref{inI}). Therefore, for any $0\leq s \leq m$, there exists some constant $C=C_0(s,b,e)$ such that
\begin{equation}\label{CosBE-1}
\left\|\widetilde{\Gamma_B}(f)\right\|_{H^{s+1}} \leq C_0(s,B,e)\,\left\|f\right\|_{H^s},\qquad \forall f \in H^s
\end{equation}
where the constant $C_0(s,B,e)$ can be chosen as
\begin{equation}\label{CosBE}
C_0(s,B,e)=C_0(s,\Phi,b) \max_{k=0,\ldots,s}\|D^k G_e\|_{L^\infty(I)}.
\end{equation}
From estimate \eqref{CosBE-1} we deduce Lemma \ref{Sob} with the following argument. Assume first $s=k \geq 1$ is an integer. Using polar coordinates
\begin{equation*}
\| \Gamma_B(f)\|_{H^k}^2=\sum_{|j| \leq k} \int_0^\infty F_j(\varrho)\varrho^2 \d\varrho \IS |\partial_v^j \widetilde{\Gamma_B}(f)(\varrho,\omega)|^2\d\omega
\end{equation*}
where, for any $|j| \leq k$, the function $F_j(\varrho)$ can be written as
\begin{equation}\label{polynome}
F_j(\varrho)=P_j(\vartheta^{(1)}(\varrho),\ldots,\vartheta^{(j)}(\varrho)) (1+\vartheta^{(1)}(\varrho))^{-n_j}.
\end{equation}
Here $P_j(y_1,\ldots,y_j)$ is a suitable polynomial, $n_j \in \mathbb{N}$ and $\vartheta^{(p)}$ denotes the $p$-th derivative of $\vartheta(\cdot)$.  Since $\vartheta \in \mathcal{C}^m( 0,\infty)$ and $I$ is a compact interval away from zero, one has $\sup_{\varrho \in I} F_j(\varrho)=C_k < \infty $ for any $|j| \leq k$.  Thus
\begin{equation}\label{tilde}
\| \Gamma_B(f)\|_{H^k} \leq C_k \|\widetilde{\Gamma_B}(f)\|_{H^{k}}
\end{equation}
where $C_k$ is an explicit constant involving the $L^\infty$ norm of the first $k$-th order derivatives of $\alpha(\cdot)$ on $I$.
This proves that the conclusion of the Lemma \ref{Sob} holds true for any integer $s \leq m$ and we deduce the general case using interpolation.
\end{proof}
\begin{nb}\label{constantCBE} It is important, for our subsequent analysis, to obtain a precise expression for the constant $C(s,B,e)$. For instance, in the case in which $e(\cdot) \in \mathcal{C}^1(0,\infty)$, one obtains that
$$
C(1,B,e)\leq  C_0(1,B,e)\,\sup_{\varrho \in I}F_1(\varrho)
$$
where  $F_1$ is of the form \eqref{polynome} with $I$ defined in \eqref{inI}. Note that $C_0(1,B,e)$ and $G_e(\varrho)$ are given by \eqref{CosBE} and \eqref{Ge} respectively.  In particular,  under Assumption \ref{HYP}, $G_e(\varrho) \leq 4\varrho$ for large $\varrho$ and $G_e(\varrho) \simeq  \varrho/2$ for $\varrho \simeq 0$.
\end{nb}
Arguing as in \cite[Corollary 3.2]{MoVi} we translate the gain of regularity obtained in Theorem \ref{theo:smooth} in gain of integrability.
\begin{cor}\label{NCe5}
Let $ {B}(u,\sigma)=\Phi(|u|)b(\widehat{u} \cdot \sigma)$ be a collision kernel satisfying \eqref{smoothphi} and $e(\cdot) \in \mathcal{C}^1(0,\infty)$ satisfying Assumption \ref{HYP}. Then, for any $1<p<\infty$
$$
\left\|\Q_{B,e}^+(f,g)\right\|_{L^p} \leq C(p,B,e)\left(\|g\|_{L^q}\,\|f\|_{L^1}+\|g\|_{L^1}\,\|f\|_{L^q}\right)
$$
where the constant $C(p,B,e)$ depends on $B$ and $e$ through the constant $C(1,B,e)$ of Theorem \ref{theo:smooth}.  The exponent $q<p$ is given by
\begin{equation}\label{expoq}
q=\left\{
\begin{array}{ccl}
\dfrac{5p}{3+2p} & \mbox{if} & p\in(1,6]\\
p/3\; & \mbox{if} & p\in[6,\infty).
\end{array}\right.
\end{equation}
\end{cor}

\subsection{Regularity and integrability for hard-spheres} We consider in this section the case of hard-spheres collision kernel $$B(u,\sigma)=\frac{|u|}{4\pi}.$$ Such a collision kernel does not enjoy the regularity properties assumed in the previous section. This does not present a problem since the dependence of the constant on the collision kernel $B$ permits to adapt the method developed in \cite{MoVi} for the elastic case.  We need some supplementary assumptions on the restitution coefficient $e(\cdot)$.
\begin{hyp}\label{HYPdiff}
In addition to the Assumptions \ref{HYP}, suppose that $e(\cdot) \in \mathcal{C}^1(0,\infty)$ and that there exists $k\in \R$ such that
$$e'(r)=O(r^k) \qquad \text{when}\ \  r \to \infty,$$
where $e'(\cdot)$ denotes the derivative of $e(\cdot)$.
\end{hyp}
The above assumptions imply that $\vartheta'(\varrho)=O(\varrho^{k+1})$ for large $\varrho$ and $\vartheta'(\varrho) \simeq 1$ when $\varrho \simeq 0$.  Recall that $\vartheta'(\cdot)$ is the derivative of $\vartheta(r)=re(r)$.
\begin{theo}\label{lpes} Assume that $e(\cdot)$ satisfies Assumptions \ref{HYPdiff}.  For any $p \in [1,3)$ there exist $\kappa >0$, $\theta \in (0,1)$ and a constant $C_e >0$ depending only on $p$ and the restitution coefficient $e(\cdot)$  such that, for any $\delta >0$
       \[ \int_{\R^3} \Q_e^+(f,f) \, f^{p-1} \d\v \le
           C_e\delta^{-\kappa} \,  \|f\|_{L^1} ^{1+p\theta} \, \|f\|_{L^p} ^{p(1-\theta)}
                         + \delta \, \|f\|_{L^1_2} \, \|f\|_{L^p _{1/p}} ^p .\]
\end{theo}
\begin{nb} The restriction $p \in [1,3)$ is the major difference with respect to the classical case \cite[Theorem 3.1]{MoVi}.  The reason is that in the inelastic regime the lack of symmetry does not permit to switch the roles of $v'$ and $v'_{\star}$, therefore, general $b$ has to be defined in the full interval $[-1,1]$.
\end{nb}
 \begin{proof} We follow the same lines presented in \cite{MoVi} and subsequently used in \cite{MiMo2}. We present the argument for convenience.  Fix $p \in [1,3)$ and let $\Theta\::\:\R \to \R^+$ be an even $\mathcal{C}^\infty$ function with compact support in $(-1,1)$  and $\int_{-1}^1 \Theta(s)\d s=1.$ In the same way, consider a radial $\mathcal{C}^\infty$ function $\Xi\::\:\R^3 \to \R$ with support in the ball $B(0,1)$ and $\IR \Xi(v)\d v=1$.  Define the mollifications $\Xi_n(v):=n^3 \Xi(nv)$ and $\Theta_m(s):=m\Theta(ms)$ for $m,n\geq 1$.  Thus, $\Phi_{S_n}=\Xi_n *(|\cdot|\chi_{A_n})$ and $b_{S_m}=\Theta_m *(\frac{1}{4\pi}\chi_{[-1 + \frac{2}{m}, 1 - \frac{2}{m}]})$ are smooth mollifications of the collision kernel.  Here we have defined the set
$$
A_n=\left\{v \in \R^3\,;\,|v| \in \left[ \frac{2}{n}, n \right]\right\} \qquad n \geq 1.
$$
Consider the smooth collision kernel
$$
B_{S_{m,n}}(|u|,\widehat{u}\cdot\sigma)=
\Phi_{S_n}(|u|)\,b_{S_m}(\widehat{u}\cdot\sigma),
$$
and observe that
$$
\mbox{supp}\left(\Phi_{S_n}\right)\subseteq\left\{\frac{1}{n} \leq |v| \leq n+1\right\}\ \ \mbox{and}\ \ \ \mbox{supp}\left(b_{S_m}\right)\subseteq\left[-1+\frac{1}{m},1-\frac{1}{m}\right].
$$
Define naturally
\begin{align*}
B_{SR_{m,n}}(|u|,\widehat{u}\cdot \sigma):&=\Phi_{S_n}(|u|)\,b_{R_m}(\widehat{u}\cdot\sigma),\\
B_{RS_{m,n}}(|u|,\widehat{u}\cdot \sigma):&=\Phi_{R_n}(|u|)\,b_{S_m}(\widehat{u}\cdot\sigma)\ \ \mbox{and}\\ B_{RR_{m,n}}(|u|,\widehat{u}\cdot \sigma):&=\Phi_{R_n}(|u|)\,b_{R_m}(\widehat{u}\cdot\sigma).
\end{align*}
Here $\Phi_{R_n}(|u|)=|u|-\Phi_{S_n}(|u|)$ and $b_{R_m}(\widehat{u}\cdot\sigma)=\frac{1}{4\pi}-b_{S_m}(\widehat{u}\cdot\sigma)$ are the remainder parts.  Thus, one splits $\Q^+_e$ in four parts using obvious notation,
$$
\Q^+_e=\Q^+_{B_{S_{m,n}},e}+\Q^+_{B_{SR_{m,n}},e}  + \Q^+_{B_{RS_{m,n}},e}  + \Q^+_{B_{RR_{m,n}},e}.
$$
Since $B_{S_{m,n}}(|u|,\widehat{u}\cdot\sigma)$ fulfills \eqref{smoothphi} one deduces from Corollary \ref{NCe5} that there is a constant $C(m,n)$ such that
$$\left\|\Q^+_{B_{S_{m,n}},e}(f,f)\right\|_{L^p} \leq C(m,n) \|f\|_{L^q}\,\|f\|_{L^1} $$
for $q < p$ given by \eqref{expoq}.  A simple application of H\"{o}lder's inequality yields
\begin{equation}\label{Q+smooth}
\int_{\R^3}  \Q^+_{B_{S_{m,n}},e}(f,f)  \, f^{p-1} \d\v \leq C(m,n)\,\|f\|_{L^q}\,\|f\|_{L^1}\,\|f\|^{p-1}_{L^p}.
\end{equation}
Recall from Corollary \ref{NCe5} that $C(m,n)$ depends on $m$ and $n$ through  the constant $C(1,B_{S_{m,n}},e)$ in Theorem \ref{theo:smooth}. Moreover, according to Remark \ref{constantCBE}, one sees that
$$
C(1,B_{S_{m,n}},e) \leq C_0(1,\Phi_{S_n},b_{S_m}) \max_{k=0,1}\|D^k G_e\|_{L^\infty(I)}\sup_{\varrho \in I} F_1(\varrho)
$$
where $C_0(s,\Phi,b)$ is the constant appearing in \eqref{coVil}, $G_e(\cdot)$ is given by \eqref{Ge}, and $F_1$ is of the form \eqref{polynome}.  The interval $I=I_{m,n}$ is defined in \eqref{inI} with $\delta=1/m$, $M=n+1$ and $a=1/n$
$$
I=\left(\sqrt{\frac{1}{2mn^2}},n+1\right).
$$
That $C_0(1,\Phi_{S_n},b_{S_m})$ depends on $m$ and $n$ in a polynomial way follows as in \cite{MoVi}.  Moreover, from the properties of $G_e$ given in Remark \ref{constantCBE} and the fact that $F_1(\varrho)$ is a rational function in $\vartheta'(\varrho)$, one deduces from Assumption \ref{HYPdiff} and the above expression of $I$ that there exist $a,b >0$ such that
\begin{equation}\label{Cmn}C(m,n)=O(m^a\,n^b)  \: \text{ as } m, n \to \infty.\end{equation}
Now, applying  Theorem \ref{alo} with $k=1$ and $\eta=-1/p'$, we get
$$\left\|\Q^+_{B_{SR_{m,n}},e}(f,f)\right\|_{L^p_\eta} + \left\|\Q^+_{B_{RR_{m,n}},e}(f,f)\right\|_{L^p_\eta} \leq \varepsilon_0(m,n)\|f\|_{L^1_{1}} \,\|f\|_{L^p_{1/p}}$$
where $\varepsilon_0(m,n)=\mathbf{C}_{-1/p',p,1}(B_{SR_{m,n}})+{\mathbf{C}}_{-1/p',p,1}(B_{RR_{m,n}})$ for any $m,n\geq1$.  In particular, using the expression of the above constants in \eqref{C_k(B)}, there exists a constant $c>0$ such that $\varepsilon_0(m,n) \leq c\, \gamma(-1/p',p,b_{R_m})=:\varepsilon(m)$ for any $m,n \geq 1$.  Then there exists some $r >0$ such that
\begin{equation}\label{espilonm}\varepsilon(m)=O(m^{-r}) \text{ as } m \to \infty.\end{equation} Indeed, since $1 \leq p < 3$, one sees from \eqref{C_k(B)} that $\gamma(-1/p',p,b_{R_m}) \leq C \|b_{R_m}\|_{L^q( \mathbb{S}^2)}$ for any $q$ such that $1 < q'< 2p'/3$.  Thus, one can choose a regularizing function $\Theta$ so that the $L^q(\mathbb{S}^2)$-norm of $b_{R_m}$  decays algebraically to zero as $m$ grows. Using the above estimate with $\eta=-1/p'$, we get
\begin{equation}\label{Q+snonsmooth}
\int_{\R^3} \left[\Q^+_{B_{SR_{m,n}},e}(f,f) + \Q^+_{B_{RR_{m,n}},e}(f,f)\right] \, f^{p-1} \d\v \leq \varepsilon(m)\|f\|_{L^1_1}\,\|f\|^p_{L^p_{1/p}}.
\end{equation}
It remains only to estimate
$$
\mbox{I}:=\ds \int_{\R^3} \Q^+_{B_{RS_{m,n}},e}(f,f)   \, f^{p-1} \d\v.
$$
One notes that
$$
\Phi_{R_n}(|v-\vb|) \leq Cn^{-1}\left(|v|^2+|\vb|^2\right),\qquad \forall v,\vb \in \R^3
$$
for some $C >0$. Thus,
$$
\mbox{I} \leq Cn^{-1}\IRR f(v)f(\vb)\left(|v|^2+|\vb|^2\right)\d v \d\vb\IS f^{p-1}(v')b_{S_m}(\widehat{u}\cdot\sigma)\d\sigma.
$$
Define
\begin{align*}
\mbox{I}_1:&=\IRR f(v)f(\vb) |v|^2 \d v \d\vb\IS f^{p-1}(v')b_{S_m}(\widehat{u}\cdot\sigma)\d\sigma, \ \ \mbox{and}\\
\mbox{I}_2:&=\IRR f(v)f(\vb) |\vb|^2 \d v \d\vb\IS f^{p-1}(v')b_{S_m}(\widehat{u}\cdot\sigma)\d\sigma.
\end{align*}
Observe that $\mbox{I}_1$ can be written as
$$
\mbox{I}_1=\IRR \Q_{B_m,e}^+(F,f)(v)\psi(v)\d v
$$
where
$$
F(v)=|v|^2 f(v), \qquad \psi(v)=f^{p-1}(v) \in L^{p'}(\R^3)
$$
with the collision kernel $B_m(|u|,\widehat{u}\cdot\sigma)=b_{S_m}(\us)$. Applying Theorem \ref{alo} with $\eta=k=0$ gives
\begin{align*}
\mbox{I}_1 &\leq \left\|\Q_{B_m,e}^+(F,f)\right\|_{L^p}\|\psi\|_{L^{p'}}\\ &\leq\mathbf{C}_{0,p,0}(B_m) \,\|F\|_{L^1}\|f\|_{L^p}\|\psi\|_{L^{p'}}\leq \mathbf{C}_{0,p,0}(B_m)\|f\|_{L^1_2}\|f\|^p_{L^p}
\end{align*}
where $\mathbf{C}_{0,p,0}(B_m)$ is defined by \eqref{C_k(B)}. Now, with the same notation,
$$
\mbox{I}_2=\IRR \Q_{B_m,e}^+(f,F)(v)\psi(v)\d v,
$$
therefore, applying Theorem \ref{alo} with $\eta=0$ and $k=-2$ yields
$$
\mbox{I}_2 \leq \mathbf{C}_{0,p,-2}(B_m) \,\|f\|_{L^1_2} \,\|F\|_{L^p_{-2}}\, \|\psi\|_{L^{p'}}\leq \mathbf{C}_{0,p,-2}(B_m) \,\|f\|_{L^1_2} \,\|f\|^p_{L^p}.
$$
Combining the two estimates for $\mbox{I}_1$ and $\mbox{I}_2$,
$$
\mbox{I} \leq \dfrac{C(m)}{n}\|f\|_{L^1_2} \,\|f\|^p_{L^p}
$$
where $C(m)=\mathbf{C}_{0,p,0}(B_m)+\mathbf{C}_{0,p,-2}(B_m)$. The support of $b_{S_m}(s)$ lies to a positive distance, of order $1/m$, from $s=1$.  Then, we use the expression \eqref{C_k(B)} to conclude that
\begin{equation}\label{C(m)}
C(m) \leq m^{-\frac{3}{2p'}} \text{ as } m \to \infty.
\end{equation}
Estimates \eqref{C(m)}, \eqref{Q+smooth} and \eqref{Q+snonsmooth} gives
\begin{multline*}
 \int_{\R^3} \Q^+_{e}(f,f)   \, f^{p-1} \d\v \leq C(m,n)\,\|f\|_{L^q}\,\|f\|_{L^1}\,\|f\|^{p-1}_{L^p} +\\
 +\varepsilon(m)\|f\|_{L^1_1}\,\|f\|^p_{L^p_{1/p}}+\dfrac{C(m)}{n} \|f\|_{L^1_2} \,\|f\|^p_{L^p}.
\end{multline*}
Using the polynomial bounds \eqref{Cmn}, \eqref{espilonm} and \eqref{C(m)} this leads to the result as in \cite{MiMo2}.
\end{proof}
\begin{nb} Assumption \ref{HYPdiff} allows to present the explicit dependence of the constants with respect to $\delta >0$.  This dependence will be crucial in the proof of Haff's law in Section 5.  Note that the constant $C_e$ in Theorem \ref{lpes} depends on the regularity of the restitution coefficient away from zero.
\end{nb}
\begin{cor}\label{lpeseta} Assume that $e(\cdot)$ satisfies Assumption \ref{HYPdiff}.  For any $p \in [1,3)$ there exist $\kappa >0$, $\theta \in (0,1)$ and a constant $C_e >0$ depending only on $p$ and the restitution coefficient $e(\cdot)$  such that, for any $\delta >0$
\begin{equation*}
\int_{\R^3} \Q_e^+(g,g) \, g^{p-1} \langle v \rangle^{\eta p}\d\v \le
C_e \delta^{-\kappa} \,  \|g\|_{L^1_\eta} ^{1+p\theta} \, \|g\|_{L^p_\eta} ^{p(1-\theta)} + \delta \, \|g\|_{L^1_{2+\eta}} \, \|g\|_{L^p _{\eta + 1/p}} ^p, \qquad \forall \eta \geq 0.
\end{equation*}
The constant $C_e$ is provided by Theorem \ref{lpes}.
\end{cor}
\begin{proof} Fix $g\geq0$, $\eta \geq 0$ and set $f(v)=g(v)\langle v \rangle^\eta$. Note that $\langle v'\rangle^\eta \leq \langle v \rangle^\eta\,\langle \vb \rangle^\eta$ for any $v,\vb \in \R^3$, then, using the weak formulation of $\Q^+_e$
\begin{equation*}
\int_{\R^3} \Q_e^+(g,g ) \, g^{p-1} \langle v \rangle^{\eta p}\d\v=\int_{\R^3} \langle v \rangle^{\eta }\Q_e^+(g,g) \, f^{p-1} \d\v \leq \IR \Q^+_e(f,f) f^{p-1}\d\v.
\end{equation*}
Conclude with Theorem \ref{lpes}.
\end{proof}
The following result applies to the rescaled solutions $g(\tau,w)$.  Its importance lies in that the estimate is uniform  in the rescaled time $\tau$.
\begin{cor}\label{lpestau}
Assume that  $e(\cdot)$ satisfies Assumption \ref{HYPdiff}.  For any $\tau \geq 0$, let $\widetilde{e}_\tau$ be the restitution coefficient defined by \eqref{eqB} and let $\Q_{\widetilde{e}_\tau}(f,f)$ be the associated collision operator. Assume that $V(\zeta(\tau))$ is continuous and goes to infinity as $\tau\rightarrow\infty$. For any $p \in [1,3)$ there exist $\kappa >0$, $\theta \in (0,1)$ and $K >0$ all independent of $\tau$ such that, for any $\delta >0$
\begin{equation*}
 \int_{\R^3} \Q_{\widetilde{e}_\tau}^+(g,g) \, g^{p-1} \langle w \rangle^{\eta p}\d\w \le K \delta^{-\kappa}\|g\|_{L^1_\eta} ^{1+p\theta} \, \|g\|_{L^p_\eta} ^{p(1-\theta)}
 + \delta \, \|g\|_{L^1_{2+\eta}} \, \|g\|_{L^p _{\eta + 1/p}} ^p, \qquad \forall \eta \geq 0.
\end{equation*}
\end{cor}
 \begin{proof} From Corollary \ref{lpeseta}, for any $\tau \geq 0$ there exists $K(\tau)=C_{\widetilde{e}_\tau}$ for which the above inequality holds.  It suffices to prove that $K =\sup_{\tau \geq 0} K(\tau) < \infty$.  Recall that $K(\tau)$ depends on $\tau$ through the restitution coefficient $\widetilde{e}_\tau$, more precisely, $C_{\widetilde{e}_\tau}$ depends on the $L^\infty$ norm of the derivatives  $D^k \widetilde{e}_\tau(\cdot)$, $k=0,1$, over some compact interval of $(0,\infty)$ bounded away from zero (independent of $\tau$). Now, for any $\tau \geq 0$,  $$D^k \widetilde{e}_\tau(\cdot)=\mu^{-k}(\tau)(D^k e)\left(\frac{\cdot}{\mu(\tau)}\right)$$ with $\mu(\tau)=V(\zeta(\tau))$. Since $\mu^{-1}(\tau)$ is continuous and goes to zero as $\tau$ goes to $\infty$, one concludes that all the $L^\infty$ norms of $D^k \widetilde{e}_\tau(\cdot)$ remain uniformly bounded with respect to $\tau$. The same holds for $K(\tau)$.\end{proof}

\section{Generalized Haff's law continued}\label{Haffcont}

\subsection{Proof of Haff's law} In this section we prove the second part of Haff's law establishing the lower bound of the temperature \eqref{converseE}. Recall that, from Theorem \ref{momlam} it suffices to prove \eqref{lam}. As explained in Section \ref{sec:haff} this is done using suitable $L^p$ estimates in the self-similar variables. In this section, the restitution coefficient fulfills Assumptions \ref{HYP2} and \ref{HYPdiff} and the collision kernel is that of hard-spheres interactions. Recall that the rescaled function $g(\tau,w)$ is solution to the Boltzmann equation in rescaled variables \eqref{eqgt}
\begin{equation}\label{cauchg}
\partial_\tau g(\tau,w) +
\xi(\tau) \nabla_w \cdot (w g(\tau,w)) =\Q_{\widetilde{e}_{\tau}}(g,g)(\tau,w) \qquad \tau >0.
\end{equation}
The restitution coefficient $\widetilde{e}_\tau$ and the time-depending mapping $\xi(\tau)$ are given by \eqref{xitau}.
\begin{propo}\label{Lpgp}
Assume that $e(\cdot)$ fulfills Assumptions \ref{HYP2} with $\gamma >0$ and \ref{HYPdiff}. Let $f_0$ satisfying \eqref{initial} with $f_0\in L^{1}_{2}\cap L^p(\R^3)$ for some $1 < p < 3$.  Let $g(\tau,\cdot)$ be the solution to the rescaled equation \eqref{cauchg} with initial datum $g(0,w)=f_0(w)$. Then, there exist $C_0 >0$ and $\kappa_0 >0$ such that
\begin{equation}\label{gp}\left\|g(\tau)\right\|_{L^p} \leq C_0(1+\tau)^{\kappa_0} \qquad \forall \tau \geq 0.\end{equation}
Consequently, there exist $C_1 >0$ and $\kappa_1 >0$ such that
\begin{equation}\label{Thettau}\mathbf{\Theta}(\tau):=\IR g(\tau,w)|w|^2\d w \geq C_1 (1+\tau)^{-\kappa_1} \qquad \forall \tau \geq 0.\end{equation}
\end{propo}
\begin{proof} The proof relies on Corollary \ref{lpestau}. Multiply \eqref{cauchg} by $g^{p-1}$ and integrate over $\R^3$ to obtain
\begin{multline}\label{gtaup}
\dfrac{1}{p}\dfrac{\d \left\|g(\tau)\right\|^{p}_{L^p}}{\d\tau}+3\left(1-\frac{1}{p}\right)\xi(\tau)\left\|g(\tau)\right\|^{p}_{L^{p}}\\=\int_{\R^3} \Q^+_{\widetilde{e}_\tau}(g,g) g^{p-1} \d w - \int_{\R^3}\Q^-(g,g) g^{p-1} \d w.
\end{multline}
From Jensen's equality, one has
\begin{equation}\label{Q-inf}
\int_{\R^3}\Q^-(g,g) g^{p-1} \d w \geq \IR g^p(\tau,w)|w|\d w \qquad \forall \tau \geq 0.
\end{equation}
According to Corollary \ref{lpestau} there exist $\kappa >0$, $\theta \in (0,1)$ and a constant $K >0$ that does not depend on $\tau$ such that
\begin{equation*}
\int_{\R^3} \Q_{\widetilde{e}_\tau}(g,g) \, g^{p-1}  \d\w \le
K \delta^{-\kappa}\|g(\tau)\|_{L^1} ^{1+p\theta} \, \|g(\tau)\|_{L^p} ^{p(1-\theta)} + \delta \, \|g(\tau)\|_{L^1_{2}} \, \|g(\tau)\|_{L^p _{1/p}} ^p, \quad \forall \delta >0.
\end{equation*}
From conservation of mass $\|g(\tau)\|\equiv1$, furthermore, $M_2:=\sup_{\tau \geq 0} \left\|g(\tau)\right\|_{L^1_{2}} <\infty$ from \eqref{halfgtau}.  Thus, using \eqref{gtaup} and \eqref{Q-inf},
\begin{multline}\label{gtaup2}
 \dfrac{\d \left\|g(\tau)\right\|^{p}_{L^p}}{\d\tau} \leq p\,K \delta^{-\kappa} \, \|g(\tau)\|_{L^p} ^{p(1-\theta)} +\,p\,M_2\,\delta\,\|g(\tau)\|_{L^p_{1/p}}^p -\mu(\tau) \|g(\tau)\|_{L^p _{1/p}} ^p
\end{multline}
where $\mu(\tau)=\min\left(1,3(p-1)\xi(\tau)\right)$. Since $\xi(\tau) \to 0$ as $\tau \to \infty$ for $\gamma >0$, there exists $\tau_0 >0$ such that
$$
\mu(\tau)=3(p-1)\xi(\tau)=\frac{3(p-1)}{\gamma \tau+1+\gamma} \quad \text{ for any } \tau \geq \tau_0.
$$
Choosing $\delta=\mu(\tau)/(pM_2)$ in \eqref{gtaup2} we get
$$
\dfrac{\d \left\|g(\tau)\right\|^{p}_{L^p}}{\d\tau} \leq p\,K\, (pM_2)^\kappa \mu(\tau)^{-\kappa} \|g(\tau)\|_{L^p} ^{p(1-\theta)} \leq C (\gamma \tau + 1+\gamma)^\kappa \|g(\tau)\|_{L^p} ^{p(1-\theta)} \quad \forall \tau \geq \tau_0
$$
for some positive constant $C>0$.  Integrating the above estimate, we conclude the existence of some constant $C_0 >0$ such that
$$\|g(\tau)\|^{p}_{L^p} \leq C_0 \left(\gamma \tau + 1+\gamma\right)^{\tfrac{\kappa+1}{\theta}} \qquad \forall \tau \geq \tau_0,$$
and \eqref{gp} readily follows.\medskip

\noindent Regarding estimate \eqref{Thettau} note that for any $R >0$,
\begin{equation*}
\begin{split}
\mathbf{\Theta}(\tau)&=\int_{|w|\leq R} g(\tau,w)|w|^2\d w + \int_{|w| >R}g(\tau,w)|w|^2\d w \\ &\geq R^2 \int_{|w| > R} g(\tau,w)\d w
\geq R^2 \left(1 -\int_{|w|\leq R} g(\tau,w)|w|\d w\right) \qquad \forall \tau \geq 0,
\end{split}
\end{equation*}
From Holder's inequality,
$$\int_{|w|\leq R} g(\tau,w)|w|\d w \leq \left(\frac{4}{3}\pi R^{3}\right)^{1/p'}\|g(\tau)\|_{L^p}\ \ \mbox{with the convention}  \ \ \frac{1}{p}+\frac{1}{p'}=1.$$
Therefore, using \eqref{gp}, there exists a positive constant $C >0$ independent of $R$ such that
$$
\mathbf{\Theta}(\tau) \geq R^2\left(1-C\,R^{3/p'} (1+\tau)^{\kappa_0}\right) \qquad \forall R >0, \qquad \forall \tau \geq 0.
$$
Pick $R=R(\tau) >0$ such that $C\,R^{3/p'} (1+\tau)^{\kappa_0}=1/2$, then
$$
\mathbf{\Theta}(\tau) \geq \frac{1}{2}R^2(\tau)=\frac{1}{2}\left(\dfrac{1}{2C(1+\tau)^{\kappa_0}}\right)^{p'/3} \qquad \forall \tau \geq 0,
$$
which gives \eqref{Thettau} with $\kappa_1=p'\kappa_0/3$.
\end{proof}
The generalized Haff's law is a consequence Theorem \ref{momlam} and Proposition \ref{Lpgp}.
\begin{theo}\label{haff} Let $f_0 \geq 0$ satisfy the conditions given by \eqref{initial} with $f_0 \in L^{p_0}(\R^3)$ for some $1 < p_0 < \infty$.  In addition, assume that $e(\cdot)$ fulfills Assumptions \ref{HYP2} and  \ref{HYPdiff}.  Then, the solution $f(t,v)$ to the associated Boltzmann equation \eqref{cauch} satisfies the generalized Haff's law
\begin{equation}\label{Haff's}
c (1+t)^{-\frac{2}{1+\gamma}}  \leq \E(t) \leq  C (1+t)^{-\frac{2}{1+\gamma}}, \qquad t \geq 0
\end{equation}
where $c,C$ are positive constants depending only on $e(\cdot)$ and $\E(0)$.
\end{theo}
\begin{proof} The upper bound in \eqref{Haff's} has already been obtained in Theorem \ref{prop:cool}. The proof of the lower bound is a straightforward consequence of Theorem \ref{momlam} and Proposition \ref{Lpgp}.  Indeed, notice that if $f_0 \in L^1(\R^3) \cap L^{p_0}(\R^3)$ for some $1 < p_0 < \infty$, using interpolation, we may assume without loss of generally that $p_0 \in (1,3)$. Recall that for $\gamma >0$,
$$
\E(t)=V^{-2}(t)\mathbf{\Theta}(\tau(t))
$$
where $V(t)=(1+t)^{\frac{1}{1+\gamma}}$ and $\tau(t)$ is given by \eqref{exptau}. Since $\mathbf{\Theta}(\cdot)$ decays at least algebraically \eqref{Thettau}, one recognizes that there exists some constant $a>0$ such that $\E(t) \geq a\left(1+t\right)^{-\mu}$ with $\mu=\frac{2+\gamma \kappa_1}{1+\gamma}$ with $\kappa_1$ being the rate in \eqref{Thettau}.  The result follows from Theorem \ref{momlam}. The proof for $\gamma=0$ is identical.
\end{proof}
\begin{exa} For constant restitution coefficient $\gamma=0$, we recover the classical Haff's law of \cite{haff} proved recently in \cite{MiMo2}:
$$
c (1+t)^{-2}  \leq \E(t) \leq  C (1+t)^{-2}, \qquad t \geq 0.
$$
\end{exa}
\begin{exa} For viscoelastic hard-spheres given in Example \ref{exa:visco} one has $\gamma=1/5$.  Thus, Theorem \ref{haff} provides the first rigorous justification of the cooling rate conjectured in \cite{BrPo,PoSc}:
$$
c (1+t)^{-5/3}  \leq \E(t) \leq  C (1+t)^{-5/3}, \qquad t \geq 0.
$$
\end{exa}
\begin{nb} Theorem \ref{haff} shows that the decay of the temperature is governed by the behavior of the restitution coefficient $e(r)$ for small impact. The cooling of the gases is slower for larger $\gamma$.
\end{nb}

 From the explicit rate of cooling of the temperature, one deduces the algebraic decay of any moments of the solution to \eqref{cauch}. Under the assumptions of the above Theorem \ref{haff} the $p-$moment $m_p(t)$ defined in \eqref{defmp} satisfies
\begin{equation}\label{momentHaff}
 c_p (1+t)^{-\frac{2p}{1+\gamma}} \leq \E(t)^p \leq m_p(t) \leq  \tilde{C}_p \, \E(t)^p \leq C_p (1+t)^{-\frac{2p}{1+\gamma}},  \qquad t \geq 0.
 \end{equation}
The positive constants $c_p, C_p, \tilde{C}_p$ depend on $p$, $m_{p}(0)$, $\E(0)$ and $e(\cdot)$.  The lower bound is a direct consequence of Jensen's inequality and \eqref{Haff's} while the upper bound has been established in Theorem \ref{momlam}.

\subsection{Application: Propagation of Lebesgue norms} We complement Proposition \ref{Lpgp} by proving the propagation of $L^p$-norms in the range $1\leq p< 3$  for the solution $g(\tau,w)$ satisfying the rescaled equation \eqref{cauchg}.   Thus, the method introduced in the elastic case \cite{MoVi} and later used in \cite{MiMo2} for constant restitution coefficient is extended to the case of a variable restitution coefficient satisfying Assumptions \ref{HYP2} and \ref{HYPdiff}.
\begin{lemme}\label{LRV}
Assume that the initial  $f_0 \geq 0$ satisfies the  conditions given by \eqref{initial} with $f_0 \in L^{p}(\R^3)$ for some $1<p<\infty$ and let $g(\tau,\cdot)$ be the solution to the rescaled equation \eqref{cauchg} with initial datum $g(0,w)=f_0(w)$. Then, there exists a constant $\nu_0 >0$ such that
$$
\int_{\mathbb{R}^{3}}g(\tau,\wb)|w-\wb|\d\wb\geq \max\left\{\nu_0,|w|\right\}\geq\frac{\nu_0}{2}\langle w \rangle, \qquad \forall w \in \R^3, \quad \tau >0.
$$
In particular,
\begin{equation*}
\int_{\R^3}
g^{p-1}\Q^-_e(g,g)\d\w \geq \frac{\nu_0}{2}\IR g^p(\tau,w)(1+|w|^2)^{1/2}\d w=\frac{\nu_0}{2} \left\|g(\tau)\right\|_{L^p_{1/p}}^p.
\end{equation*}
\end{lemme}
\begin{proof} The proof is a simple consequence of
$$
\mathbf{\Theta}_{\mathrm{min}}:=\inf_{\tau >0} \IR g(\tau,w)|w|^2\d w >0.
$$
Indeed, since $f_0 \in L^1_3$ the propagation of $p$-moments in the rescaled variables implies $\sup_{t\geq 0}\left\|g(\tau)\right\|_{L^{1}_{3}}<\infty$.  Then, for $R>0$ large enough
\begin{align*}
\int_{\{|w|\leq R\}}g(\tau,w)|w|^{2}\d\w&=\IR g(\tau,w)|w|^{2}\d\w-\int_{\{|w|\geq R\}}g(\tau,w)|w|^{2}\d\w\\
&\geq \mathbf{\Theta}_{\mathrm{min}} -\frac{1}{R} \sup_{\{\tau\geq0 \}}\left\|g(\tau)\right\|_{L^{1}_{3}}\geq \mathbf{\Theta}_{\mathrm{min}}/2>0.
\end{align*}
We conclude that,
$$\IR g(\tau,w)|w|\d\w \geq\frac{1}{R}\int_{\{|w|\leq R\}}g(\tau,w)|w|^{2}\d\w  \geq \dfrac{\mathbf{\Theta}_{\mathrm{min}}}{2R}=:\nu_0>0.$$
Using this observation and Jensen's inequality we obtain the result.
\end{proof}
\begin{theo}\label{Lpg}
Assume the variable restitution coefficient $e(\cdot)$ satisfy Assumptions \ref{HYP2} and \ref{HYPdiff} for some \emph{positive} $\gamma >0$. Assume that $f_0 \geq 0$ satisfies \eqref{initial} with $f_0 \in L^{1}_{2(1+\eta)}\cap L^p_{\eta}(\R^3)$ for some $1 \leq  p <3$ and $\eta \geq 0$.  Then, the rescaled solution $g(\tau,\cdot)$ to \eqref{cauchg} with initial datum $g(0,w)=f_0(w)$ satisfies
$$
\sup_{\tau \geq 0} \left\|g(\tau)\right\|_{L^p_\eta}  < \infty.
$$
In particular,
$$
\sup_{t\geq0}\left\{V(t)^{-3/p'}\left\|f(t)\right\|_{L^{p}}\right\}= \sup_{\tau \geq 0} \left\|g(\tau)\right\|_{L^p} < \infty.
$$
Recall that $V(t)=(1+t)^{\frac{1}{1+\gamma}}$.
\end{theo}
\begin{proof}
Multiplying equation \ref{Lpgp} by $g^{p-1}(\tau,w)\left\langle w\right\rangle^{\eta p}$ and integrating over $\R^3$ yields
\begin{multline*}
\dfrac{1}{p}\dfrac{\d \left\|g(\tau)\right\|^{p}_{L^p_\eta}}{\d\tau}+3\left(1-\frac{1}{p}\right)\xi(\tau)\left\|g\right\|^{p}_{L^{p}_\eta}=\int_{\R^3} \Q^+_{\widetilde{e}_\tau}(g,g) g^{p-1} \left\langle w\right\rangle^{\eta p}\d w -\\ \int_{\R^3}\Q^-(g,g) g^{p-1} \left\langle w\right\rangle^{\eta p}\d w+\eta\xi(\tau)\int_{\R^3}g^{p}(\tau,w)|w|^{2}\left\langle w\right\rangle^{\eta p-2}\d w.
\end{multline*}
Using Lemma \ref{LRV} one has
$$
\int_{\R^3}\Q^-(g,g) g^{p-1} \left\langle w\right\rangle^{\eta p}\d w \geq \dfrac{\nu_0}{2}\|g(\tau)\|_{L^p_{\eta+1/p}}^p.
$$
Moreover, $ C_\eta =\sup_{\tau \geq 0} \left\|g(\tau)\right\|_{L^1_{2+\eta}} <\infty$ by virtue of the propagation of moments in self-similar variables \eqref{momentHaff}. Applying Corollary \ref{lpestau} with $\delta=\frac{\nu_0}{4C}$,
\begin{multline}\label{pds}
\dfrac{1}{p}\dfrac{\d }{\d \tau}\left\|g(\tau)\right\|^{p}_{L^p_\eta}+\frac{\nu_0}{4}\left\|g(\tau)\right\|^{p}_{L^{p}_{\eta+1/p}}\\ \leq K \left\|g(\tau)\right\|^{p(1-\theta)}_{L^{p}_\eta}+\xi(\tau)\left(\eta-\frac{3}{p'}\right)\left\|g(\tau)\right\|_{L^p_\eta}^p \quad\forall\tau>0
\end{multline}
for some uniform constant $K$. Since $\gamma >0$, the mapping $\xi(\tau)$ decreases toward zero, thus \eqref{pds} leads to the result.
\end{proof}
\begin{nb} We refer to \cite[Theorem 1.3]{MiMo2} for a proof of the case $\gamma=0$. Furthermore, additional pointwise estimates allow to extend the above result to $p \geq 3$ assuming higher moments for $f_0$. We refer the reader to \cite{AloLo} for further similar estimates.
\end{nb}
\section{High-energy tails for the self-similar solution}\label{tails}

We finalize this work studying the high-energy tails of $f(t,v)$ of the solution to \eqref{be:force}.  For models with variable restitution coefficient the high energy tail is dynamic since gas changes its behavior during the cooling process.  This is noted with a dynamic rate in the tail.  Here again, we shall deal with the generalized hard-spheres collision kernel
$$
B(u,\sigma)=|u|b(\widehat{u}\cdot \sigma)
$$
where $b(\cdot)$ satisfies \eqref{normalization}.  We argue in the self-similar variables, thus it is convenient to define the rescaled $p$--moments
$$
\mathbf{m}_p(\tau)=\IR g(\tau,w)\,|w|^{2p}\d w,\ \ \ p\geq0.
$$
Notice that (\ref{momentHaff}) readily translates into
\begin{equation}\label{momentself}
c_p\leq \mathbf{m}_p(\tau) \leq C_p \quad \mbox{for}\ \ \tau\geq0.
\end{equation}
The following Theorem generalizes \cite[Proposition 3.1]{MiMo2} to the case of a variable restitution coefficient.
\begin{theo}[\textbf{$L^{1}$-exponential tails Theorem}]\label{intexpbounds}
Let $B(u,\sigma)=|u|b(\us)$ satisfy \eqref{normalization} with $b\in L^{q}(\mathbb{S}^{2})$ for some $q>1$.  Assume that $e(\cdot)$ and $f_0$ fulfill Assumptions \ref{HYP2} and \eqref{initial} respectively.  Furthermore, assume that there exists $r_0 >0$ such that
\begin{equation*}
\IR f_0(v)\exp\left(r_0|v|\right) \d v <\infty.
\end{equation*}
Let $g(\tau,w)$ be the rescaled solution defined by \eqref{resca}. Then, there exists some $r\leq r_0$ such that
\begin{equation}\label{expbounds}
\sup_{\tau\geq0}\IR g(\tau,w)\exp\left(r|w|\right) \d w < \infty.
\end{equation}
Consequently,
\begin{equation}\label{expboundssol}
\sup_{t\geq 0}\IR f(t,v)\exp\left(rV(t)|v|\right)\d w < \infty.
\end{equation}
\end{theo}
\begin{proof}
The method of proof is carefully documented in \cite{AloGam,BoGaPa}.  We sketch the proof dividing the argument in 5 steps.\medskip

\indent \textit{Step 1.} Note that formally
\begin{equation*}
\IR g(\tau,w)\exp\left(r|w|^{s}\right)\d w=\sum^{\infty}_{k=0}\frac{r^{k}}{k!}\mathbf{m}_{sk/2}(\tau),
\end{equation*}
for any $r >0$ and any $s >0.$ Hence, the summability of the integral is described by the behavior of the functions $\frac{\mathbf{m}_{sk/2}(\tau)}{k!}$.  This motivates the introduction of the renormalized moments
$${z}_p(\tau):=\frac{\mathbf{m}_p(\tau)}{\Gamma(ap+b)}, \ \ \mbox{with}\ \ a=2/s,$$
where $\Gamma(\cdot)$ denotes the Gamma function. We shall prove that the series converges for some $r < r_0$ and with $s=1$ (i.e. $a=2$). To do so, it is enough to prove that, for some $b < 1$ and $Q >0$ large enough, one has $z_p(\tau)\leq Q^p$ for any $p \geq 1$ and any $\tau \geq0$.\medskip

\indent \textit{Step 2.} Recall that, according to Lemma \ref{gamm}, the estimates of Proposition \ref{povzner} are independent of the restitution coefficient $e(\cdot)$. In particular, they  hold for the time-dependent collision operator $\Q_{\widetilde{e}_\tau}$ providing bounds which are \textit{uniform} with respect to $\tau$. Specifically,
\begin{equation*}
\IR \Q_{\widetilde{e}_\tau}(g,g)(\tau, w)|w|^{2p}\d w\leq-(1-\kappa_{p})\mathbf{m}_{p+1/2}(\tau) +\kappa_{p}\; \mathcal{S }_{p}(\tau), \qquad \forall \tau \geq 0
\end{equation*}
where $\kappa_p$ is the constant introduced in Lemma \ref{gamm} and \begin{equation*}
\mathcal{S}_{p}(\tau)=\sum^{[\frac{p+1}{2}]}_{k=1}\left(
\begin{array}{c}
p\\k
\end{array}
\right)\left(\mathbf{m}_{k+1/2}(\tau)\;\mathbf{m}_{p-k}(\tau)+\mathbf{m}_{k}(\tau)\;\mathbf{m}_{p-k+1/2}(\tau)\right).
\end{equation*}
\indent \textit{Step 3.}  An important simplification, first observed in \cite{BoGaPa}, consists in noticing that the term $\mathcal{S}_p$ satisfies
\begin{equation*}
\mathcal{S}_{p}(\tau)\leq A\;\Gamma(ap+a/2+2b)\;\mathcal{Z}_{p}(\tau)\ \ \mbox{for}\ \ \ a\geq1,\;b>0,
\end{equation*}
where $A=A(a,b) >0$ does not depend on $p$ and
\begin{equation*}
\mathcal{Z}_p(\tau)=\max_{1\leq k\leq k_p}\left\{z_{k+1/2}(\tau)\;z_{p-k}(\tau),z_k(\tau)\;z_{p-k+1/2}(\tau)\right\}.
\end{equation*}
With such an estimate, the rather involved term $\mathcal{S}_p$ is more tractable.\medskip

\indent \textit{Step 4.} Using the above steps and the evolution problem \eqref{cauchg} satisfied by the rescaled solution $g$, we check that
\begin{equation*}
\frac{\d \mathbf{m}_p}{\d\tau}(\tau)+(1-\kappa_p)\mathbf{m}_{p+1/2}(\tau)\leq\kappa_p\;\Gamma(ap+a/2+2b)\mathcal{Z}_p(\tau)+2p\,\xi(\tau)\mathbf{m}_p(\tau)
\end{equation*}
where we used the fact that
$$
\IR |w|^{2p}\nabla_w \cdot (w g (\tau,w)) \d w=-2p\,\mathbf{m}_p(\tau).
$$
Using the asymptotic formula
$$
\lim_{p\rightarrow\infty}\frac{\Gamma(p+r)}{\Gamma(p+s)}p^{s-r}=1,
$$
the fact that $\xi(\tau)\leq1$ and $\kappa_p\sim1/p^{1/q'}$ for large $p$, one concludes that there are constants $c_{i}>0$ ($i=1,2$) and $p_0 >1$ sufficiently large so that
\begin{equation*}
\frac{\d z_p}{\d\tau}(\tau)+c_1\;p^{a/2}z^{1+1/2p}_{p}(\tau)\leq c_2\;p^{a/2+b-1/q'}\;\mathcal{Z}_p(\tau)+2p\;z_p(\tau) \qquad \forall \tau \geq 0, \:p \geq p_0.
\end{equation*}
We also used that $\mathbf{m}_{p+1/2}(\tau) \geq \mathbf{m}_p^{1+1/2p}(\tau)$ for any $\tau \geq 0 $ thanks to Jensen's inequality.\medskip

\indent \textit{Final step.} We claim that if we choose $a=2$ and $0<b<1/q'$ it is possible to find $Q>0$ large enough so that $\mathbf{m}_p(\tau)\leq Q^{p}$.  Indeed, let $p_0$ and $Q<\infty$ such that
\begin{equation*}
\frac{c_2}{c_1}p^{b-1/q'}_0\leq\frac{1}{2},\ \ \mbox{and} \ \ Q\geq\left\{\max_{1\leq k\leq p_0}\sup_{\tau\geq0} z_k(\tau),Q_0,\frac{16}{c^{2}_1},1\right\},
\end{equation*}
where $Q_0$ is a constant such that $z_p(0)\leq Q_0^{p}$.  This constant exists by the exponential integrability assumption on the initial datum.  Moreover, since moments of $g$ are uniformly propagated, the existence of such \textit{finite} $Q$ is guaranteed. Arguing by induction and standard comparison of ODE's, one proves that  $y_p(\tau):=Q^{p}$ satisfies for $p\geq p_0$
\begin{equation*}
\frac{\d y_p}{\d\tau}(\tau)+c_1\;p^{a/2}y_p^{1+1/2p}(\tau)\geq c_2\;p^{a/2+b-1/q'}\;\mathcal{Z}_p(\tau)+2p\;y_p(\tau), \quad y_p(0)\geq z_p(0)
\end{equation*}
therefore, $y_p(\tau)\geq z_p(\tau)$ for any $p\geq p_0.$ Since this is trivially true for $p<p_0$ we obtain that
$$\mathbf{m}_p(\tau) \leq \Gamma(2p+b) Q^p, \qquad \forall p \geq 1, \tau \geq 0.$$
From Step 1, this is enough to prove the Theorem.\end{proof}
\begin{exa} For viscoelastic hard-spheres $V(t)=(1+t)^{5/3}$. Therefore,
\begin{equation*}
\IR f_0(v)\exp\left(r_0|v|\right) \d v <\infty \Longrightarrow \sup_{t\geq 0}\int_{\R^3}f(t,v)\exp\left(r(1+t)^{5/3}|v|\right)\d v< \infty
\end{equation*}
for some  $r < r_0$.
In particular, using the terminology of \cite{BoGaPa}, $f(t,v)$ has a (dynamic) exponential tail of order 1.
\end{exa}

\section*{Appendix A: Viscoelastic hard-spheres}\setcounter{equation}{0}
\renewcommand{\theequation}{A.\arabic{equation}}

In this Appendix we prove that Assumptions \ref{HYP2} are met by the restitution coefficient $e(\cdot)$ associated to the so-called viscoelastic hard-spheres as derived in \cite{PoSc} (see also \cite[Chapter 4]{BrPo}). In fact, we prove a more general result for the hard-spheres collision kernel
$$
B(u,\sigma)=\dfrac{|u|}{4\pi} \qquad \forall u \in \R^3,\:\sigma \in \mathbb{S}^2.
$$
Recall that $\mathbf{\Psi}_e$ was defined in \eqref{Psie} as
$$\mathbf{\Psi}_e(x)=\dfrac{1}{2\sqrt{x}}\int_0^{\sqrt{x}} \left(1-e(z)^2\right)z^3\d z, \qquad x >0.
$$
\begin{lemmeA}\label{decreas} Assume that $e(\cdot)$ satisfies Assumption \ref{HYP} and that the mapping $r \geq 0 \mapsto e(r)$ is decreasing. Then, the associated function $\mathbf{\Psi}_e$ defined in \eqref{Psie} is strictly increasing and convex.
\end{lemmeA}
\begin{proof} Since $e$ is decreasing, $e'(r) \leq 0$ for any $r \geq 0$.  Here $e'(\cdot)$ denotes the derivative of $e(\cdot)$. Define
$$
\Phi(x):=\dfrac{1}{x}\int_0^{x} \left(1-e^2(z)\right)z^3\d z, \qquad x >0.
$$
Note that $\mathbf{\Psi}_e(\cdot)$ is convex if and only if $x \Phi_{xx}(x) -\Phi_x (x) \geq 0$ for any $x >0$ where $\Phi_{x}$ and $\Phi_{xx}$ denote the first and second derivatives of $\Phi$ respectively. A simple calculation shows that
$$
x \Phi_{xx}(x) -\Phi_x (x)=-2 x^3 e'(x)e(x) + \dfrac{3}{x^2}\int_0^x (1-e^2(z))z^3 \d z, \qquad \forall x >0.
$$
Since $e'(x) \leq 0$ and $e(\cdot) \in (0,1]$ one concludes that $x \Phi_{xx}(x) -\Phi_x (x) \geq 0$ for any $x >0$.\medskip

Similarly, since $e'(\cdot) \leq 0$ the mapping $z \geq 0 \mapsto (1-e^2(z))z^3$ is nondecreasing, thus, $\Phi_x (x) > 0$ for any $x >0$. This implies that $\mathbf{\Psi}_e(\cdot)$ is strictly increasing over $(0,+\infty)$.
\end{proof}

For the visco-elastic hard-spheres, as derived in \cite{PoSc}, the restitution coefficient $e$ is solution of the equation
\begin{equation}\label{visco2}
e(r)+ \alpha\,r^{1/5} e(r)^{3/5}=1 \qquad \forall r \geq 0
\end{equation}
where $\alpha >0$ is a constant depending on the material viscosity.  It was proved in \cite[p. 1006]{AlonsoIumj} that, on the basis of \eqref{visco2}, Assumptions \ref{HYP} are met. From equation \eqref{visco2}, one deduces that
$$
\lim_{r \to 0^+}e(r)=1,\ \ \mbox{and}\ \ e(r) \simeq 1-\alpha r^{1/5}\ \ \text{for}\ \ r \simeq 0
$$
which means that Assumption \ref{HYP2} \textit{(1)} is met. Furthermore, equation \eqref{visco2} also implies that $e$ is continuously decreasing. According to Lemma A.\ref{decreas}, $e(\cdot)$ satisfy Assumptions \ref{HYP2}. Moreover, it is easy to deduce from \eqref{visco2} that Assumption \ref{HYPdiff} is satisfied.
\begin{exaA} For monotone decreasing restitution coefficient introduced in Example \ref{dec}, Assumptions \ref{HYP2} are also met by virtue of the above Lemma. In such a case, according to \eqref{eta}, the cooling of the temperature $\E(t)$ is
$$
\E(t)= O\left((1+t)^{-\frac{2}{1+\eta}}\right) \quad\text{as} \quad t \to \infty.
$$
\end{exaA}


\begin{thebibliography}{99}

\bibitem{AlonsoIumj}
\textsc{Alonso, R. J.,} Existence of global solutions to the Cauchy problem for the inelastic Boltzmann equation with near-vacuum data, \textit{Indiana Univ. Math. J.}, \textbf{58} (2009), 999--1022.

\bibitem{AloGam}
\textsc{Alonso, R. J. \& Gamba, I. M.,} Propagation of $L\sp 1$ and $L\sp \infty$ Maxwellian weighted bounds for derivatives of solutions to the homogeneous elastic Boltzmann equation,  \textit{J. Math. Pures Appl.} \textbf{ 89}  (2008),  575--595.


\bibitem{AloCar}
\textsc{Alonso, R. J., Carneiro, E. \& Gamba, I. M.,} Convolution inequalities for the Boltzmann
collision operator, \textit{Comm. Math. Phys.}, to appear.
\bibitem{AloLo}
\textsc{Alonso, R. J. \& Lods, B.} Work in preparation.

\bibitem{BisiCT}
{\sc  Bisi, M.,  Carrillo, J.A. \& Toscani, G.}, Contractive Metrics
for a Boltzmann equation for granular gases: Diffusive equilibria,
{\em J. Statist. Phys.} \textbf{118} (2005), 301--331.

\bibitem{BisiCT2}
 {\sc  Bisi, M.,  Carrillo, J.A. \& Toscani, G.},
Decay rates in probability metrics towards homogeneous cooling
states for the inelastic Maxwell model,  {\it J. Statist. Phys.}
\textbf{124} (2006), 625--653.

\bibitem{Bobmom}
\textsc{Bobylev, A. V.}, Moment inequalities for the Boltzmann equation and applications to spatially homogeneous problems, \textit{J. Statist. Phys.}\textbf{ 88} (1997), 1183--1214.
\bibitem{Bobylev-Carrillo-Gamba}
 {\sc Bobylev, A. V.,  Carrillo, J. A. \& Gamba,  I.},
On some properties of kinetic and hydrodynamic equations for
inelastic interactions, {\it J. Statist. Phys.} {\bf 98} (2000),
743--773; Erratum on: {\it J. Statist. Phys.} {\bf 103}, (2001),
1137--1138.

\bibitem{BoGaPa} {\sc Bobylev, A. V.,  Gamba, I. M. \&  Panferov, V.}, {Moment
inequalities and high-energy tails for the Boltzmann equations with
inelastic interactions}, \textit{J. Statist. Phys.} {\bf 116}
(2004), 1651--1682.




\bibitem{BrPo}
\textsc{Brilliantov,  N. V. \&   P\"{o}schel,  T.}, \textbf{Kinetic
theory of granular gases}, Oxford University Press, 2004.
\bibitem{ccc}
\textsc{Carlen, E. A., Carrillo, J. A. \& Carvalho, M. C.}, Strong convergence towards homogeneous cooling states
for dissipative Maxwell models, \textit{Ann. I. H. Poincaré - AN}, \textbf{26} (2009), 1675--1700.


\bibitem{CaTo}
\textsc{Carrillo, J. A. \& Toscani, G.}, Contractive probability
metrics and asymptotic behavior of dissipative kinetic equations,
\textit{Riv. Mat. Univ. Parma},  \textbf{6}  (2007), 75--198.

\bibitem{Cer}
\textsc{Cercignani, C.},
\newblock {\bf The {Boltzmann} equation and its applications},
\newblock Springer, New York, 1988.

\bibitem{GaPaVi} \textsc{Gamba, I., Panferov, V. \&  Villani, C.}, {On the
Boltzmann equation for diffusively excited granular media},
\textit{Comm. Math. Phys.} {\bf 246} (2004), 503--541.

 
\bibitem{haff}
\textsc{Haff P. K.}, Grain flow as a fluid-mechanical phenomenon, \textit{J. Fluid Mech.} \textbf{134} (1983).

\bibitem{lions} \textsc{Lions, P.-L.},  Compactness in {B}oltzmann's equation
via Fourier integral operators and applications I, II, III,
\textit{J. Math. Kyoto Univ.} {\bf 34} (1994), 391--427, 429--461,
539--584.



\bibitem{MMR}\textsc{Mischler, S., Mouhot, C. \&  Rodriguez Ricard, M.}, Cooling process for inelastic
{B}oltzmann equations for hard-spheres, Part~I: The Cauchy
problem, \textit{J. Statist. Phys.} {\bf 124} (2006), 655-702.

\bibitem{MiMo}\textsc{Mischler, S. \& Mouhot, C.}, Cooling process for inelastic {B}oltzmann equations for hard-spheres, Part~II: Self-similar solution and tail behavior,
\textit{J. Statist. Phys.} {\bf 124} (2006), 655-702.

\bibitem{MiMo2}\textsc{Mischler, S. \& Mouhot, C.}, Stability, convergence to self-similarity and elastic limit for the Boltzmann equation for inelastic hard-spheres.  \textit{Comm. Math. Phys.}  \textbf{288}  (2009),  431--502.

\bibitem{MiMo3}\textsc{Mischler, S. \& Mouhot, C.}, Stability, convergence to the steady state and elastic limit for the Boltzmann equation for diffusively excited granular media.  \textit{Discrete Contin. Dyn. Syst.  A}  \textbf{24}  (2009),   159--185.




\bibitem{MoVi} \textsc{Mouhot, C. \& Villani, C.},  Regularity theory for the
spatially homogeneous {B}oltzmann equation with cut-off,
\textit{Arch. Ration. Mech. Anal.} {\bf 173} (2004), 169--212.

\bibitem{PoSc}
{\sc Schwager, T. \& P\"{o}schel, T.}, Coefficient of normal restitution of viscous particles and cooling rate of granular gases, \textit{Phys. Rev. E} {\bf 57} (1998), 650--654.


\bibitem{Vil}
\textsc{Villani, C.}, A review of mathematical topics in collisional kinetic theory, \textit{Handbook of Mathematical
Fluid Dynamics}, Vol. I, 71--305 (North-Holland, Amsterdam, 2002).
\bibitem{Vi}
{\sc Villani, C.},
\newblock{Mathematics of granular materials,}
\newblock{\em J. Statist. Phys.} {\bf 124} (2006),  781--822.

\bibitem{Wennberg}
\textsc{Wennberg, B.}, Regularity in the Boltzmann equation and the Radon transform, \textit{Comm. Partial
Differential Equations} \textbf{19} (1994), 2057--2074.

\end{thebibliography}
\end{document}